\documentclass[a4paper]{article}
\usepackage{amsmath,bm}
\usepackage{amssymb}
\usepackage{amsthm}
\usepackage{mathtools}
\usepackage{mathrsfs}
\usepackage[utf8]{inputenc}
\usepackage{cases}
\usepackage{ulem}
\usepackage{graphicx,color}
\usepackage{kantlipsum}
\usepackage{enumitem}
\usepackage[numbers]{natbib}

\theoremstyle{plain}
\newtheorem{thm}{Theorem}[section]
\newtheorem{lem}[thm]{Lemma}
\newtheorem{cor}[thm]{Corollary}
\newtheorem{prop}[thm]{Proposition}

\theoremstyle{definition}
\newtheorem{df}[thm]{Definition}
\newtheorem{eg}[thm]{Example}

\newtheorem{asm}[thm]{Assumption}

\theoremstyle{remark}
\newtheorem{rem}[thm]{Remark}

\newcommand{\qad}{\phantom{={}}}
\newcommand{\relmiddle}[1]{\mathrel{}\middle#1\mathrel{}}
\numberwithin{equation}{section}
\allowdisplaybreaks[3]

\newcommand{\E}{\mathbb{E}}

\newcommand{\N}{\mathbb{N}}
\renewcommand{\P}{\mathbb{P}}

\newcommand{\R}{\mathbb{R}}

\newcommand{\X}{\mathbb{X}}
\newcommand{\Y}{\mathbb{Y}}
\newcommand{\Z}{\mathbb{Z}}
\newcommand{\1}{\mbox{1}\hspace{-0.25em}\mbox{l}}

\newcommand{\cC}{\mathcal{C}}

\newcommand{\cE}{\mathcal{E}}
\newcommand{\cF}{\mathcal{F}}

\newcommand{\cH}{\mathcal{H}}
\newcommand{\cI}{\mathcal{I}}
\newcommand{\cK}{\mathcal{K}}

\newcommand{\cM}{\mathcal{M}}

\newcommand{\cP}{\mathcal{P}}

\newcommand{\cV}{\mathcal{V}}

\newcommand{\cX}{\mathcal{X}}
\newcommand{\cY}{\mathcal{Y}}

\newcommand{\sF}{\mathscr{F}}

\renewcommand{\a}{\alpha}
\renewcommand{\b}{\beta}

\newcommand{\Gm}{\Gamma}
\newcommand{\dl}{\delta}
\newcommand{\Dl}{\Delta}
\newcommand{\eps}{\varepsilon}
\newcommand{\lm}{\lambda}
\newcommand{\Lm}{\Lambda}
\newcommand{\sg}{\sigma}

\newcommand{\del}{\partial}

\DeclareMathOperator{\diam}{diam}

\DeclareMathOperator{\hist}{HIST}
\DeclareMathOperator{\im}{Im}

\DeclareMathOperator{\Lip}{Lip}

\newcommand{\plim}[1][]{\mathop{\varprojlim}\limits_{#1}}

\DeclareMathOperator{\rank}{rank}

\DeclareMathOperator{\supp}{supp}

\usepackage{comment}

\title{
Large deviation principle for persistence diagrams of random cubical filtrations
}

\author{
Yasuaki \textsc{Hiraoka}\thanks{Kyoto University Institute for Advanced Study, WPI-ASHBi, Kyoto University, Kyoto 606-8501, Japan. \texttt{hiraoka.yasuaki.6z@kyoto-u.ac.jp}}
\and Shu \textsc{Kanazawa}\thanks{Kyoto University Institute for Advanced Study, Kyoto University, Kyoto 606-8501, Japan;
Department of Mathematics, The Ohio State University, Columbus 43210, USA. \texttt{kanazawa.shu.2w@kyoto-u.ac.jp}}
\and Jun \textsc{Miyanaga}\thanks{Department of Mathematics, Graduate School of Science, Kyoto University, Kyoto 606-8502, Japan. \texttt{j.miyanaga@math.kyoto-u.ac.jp}}
\and Kenkichi \textsc{Tsunoda}\thanks{Faculty of Mathematics, Kyushu University, Fukuoka 819-0395, Japan. \texttt{tsunoda@math.kyushu-u.ac.jp}}
}

\date{}

\begin{document}
\maketitle
\begin{abstract}
The objective of this article is to investigate the asymptotic behavior of the persistence diagrams of a random cubical filtration as the window size tends to infinity.
Here, a random cubical filtration is an increasing family of random cubical sets, which are the union of randomly generated higher-dimensional unit cubes with integer coordinates in a Euclidean space.
We first prove the strong law of large numbers for the persistence diagrams, inspired by the work of Hiraoka, Shirai, and Trinh, where the persistence diagram of a filtration of random geometric complexes is considered.
As opposed to prior papers treating limit theorems for persistence diagrams, the present article aims to further study the large deviation behavior of persistence diagrams.
We prove a large deviation principle for the persistence diagrams of a class of random cubical filtrations, and show that the rate function is given as the Fenchel--Legendre transform of the limiting logarithmic moment generating function.
In the proof, we also establish a general method of lifting a large deviation principle for the tuples of persistent Betti numbers to persistence diagrams for broad application.
\end{abstract}

\providecommand{\keywords}[1]
{
  \small	
  \textbf{Keywords } #1
}
\providecommand{\subjclass}[1]
{
  \small	
  \textbf{Mathematics Subject Classification } #1
}




\section{Introduction}
Many large complex systems such as social and biological networks are modeled by random graphs, where each vertex and edge represent an object and a connection between two individual objects, respectively.
Random graph theory, which dates back to the work of Erd\H os and R\'enyi~\cite{ER59,ER60}, has been providing a good understanding of such large complex systems.
However, we are often interested in relations among multiple objects rather than pairwise relations in some applications.
In such cases, \textit{simplicial complexes}, composed of various dimensional simplices (e.g. vertices, edges, solid triangles, and solid tetrahedra), are useful as natural higher-dimensional objects of graphs.
Motivated by various problems in the growing field of topological data analysis, there has been considerable study of the topology of \textit{random simplicial complexes}.

The study of random simplicial complexes has its origin in the work of Linial and Meshulam~\cite{LM06}.
They introduced a random $2$-dimensional simplicial complex, the so-called \textit{$2$-Linial--Meshulam--Wallach complex}, as a high-dimensional generalization of Erd\H os--R\'enyi graphs, and showed the phase transition of the vanishing of its first homology group with mod $2$ coefficients.
More generally, the \textit{$k$-Linial--Meshulam--Wallach complex} was introduced by Meshulam and Wallach~\cite{MW09}, and its topology has been studied, for example, in~\cite{AL16,AL15,ALLM13,BHK11,CCFK11,HKP17,KP16,Ko10,LP16,LP17,LuP18,NP18}.
Besides, there has been extensive study of other random simplicial complex models as generalizations of Erd\H os--R\'enyi graphs.
Among them, the \textit{multi-parameter random simplicial complex model}, introduced by Costa and Farber, has been drawing attention as a fairly general model (see, e.g.,~\cite{CF16a,CF17a,CF17b,CF16b,FMN22,Fo19,Ka22,OST21}).

Another type of random simplicial complex, the so-called \textit{random geometric complex}, was introduced in~\cite{{K11}} as a natural higher-dimensional generalization of random geometric graphs~\cite{P03}.
In a random geometric complex, the vertices are given as random points in a metric space, and the higher-dimensional simplices are placed according to some deterministic rule based on the proximity relationship among those points (cf. \v Cech complex, Vietoris--Rips complex, alpha complex).
Yogeshwaran--Subag--Adler~\cite{YSA17} proved the strong law of large numbers for the Betti numbers (the dimension of the homology group with real-coefficients) of a random geometric complex built over an ergodic and stationary point process in a Euclidean space as the window size tends to infinity.
See~\cite{BM15,GTT19} for manifold settings instead of Euclidean spaces.
For comprehensive surveys of random simplicial complex models, see also~\cite{BK18,BK22,K14b}.

In this article, we deal with \textit{cubical sets}, composed of various dimensional elementary cubes (e.g., lattice points, line segments, squares, and cubes) in a Euclidean space, rather than simplicial complexes.
In digital image analysis, a cubical set is exploited to represent a digital image data and to obtain information of shapes.
Particularly, the \textit{cubical homology} enables us to extract the information about holes such as loops and cavities, and provides a useful descriptor of topological features in digital images (see Subsection~\ref{ssec:cube_hom} for the definition of cubical sets and the cubical homology).
As an analog of the study of random simplicial complexes, there has been a growing interest in the study of \textit{random cubical sets}.
Hiraoka--Tsunoda~\cite{HT18} proved the strong law of large numbers for the Betti numbers of a class of random cubical sets.
See also~\cite{DL22,HS16,KPR21,WW16} for other types of studies.

On the other hand, \textit{persistent homology}~\cite{ELZ02,ZC05} is getting a lot more attention in the rapidly emerging field of topological data analysis.
Via the theory of persistent homology, we can not only capture holes in a given data but also measure the robustness of the holes.
Let us take a grayscale image data for example.
Since a grayscale image can be regarded as a function representing the intensity of light on each elementary cube, we can construct a \textit{cubical filtration}, an increasing family of cubical sets, by considering the sublevel sets of the function at varying thresholds.
Then, the persistent homology can extract the information about the birth and death times of holes in the filtration.
By plotting all birth-death pairs into a $2$-dimensional parameter space, we get a useful descriptor, called a \textit{persistence diagram}, of the multiscale topological features embedded in the grayscale image data.
A birth-death pair far from the diagonal line of persistence diagram is, in fact, robust to a perturbation of data~\cite{CEH07}, and usually considered as a characterization of data.
However, since grayscale image data usually contain measurement noise, or more importantly could be drawn from some probability distribution, it is significant to examine the effect of randomness on persistence diagrams.
With the motivation, we study the persistence diagram of a \textit{random cubical filtration}.

The objective of this work is to investigate the asymptotic behavior of the persistence diagrams of random cubical filtrations as the window size tends to infinity.
We first prove the strong law of large numbers for the persistence diagrams of a class of random cubical filtrations, which states that the persistence diagram converges vaguely to a deterministic measure almost surely (Theorem~\ref{thm:LLN_PD}).
The proof is highly inspired by the work of Hiraoka, Shirai, and Trinh~\cite{HST18}, where the persistence diagram of a filtration of random geometric complexes is considered.
As the first step of the proof, we prove the strong law of large numbers for persistent Betti numbers, roughly speaking which is the number of holes that persist over a fixed time-interval, using their nearly additive property (Theorem~\ref{thm:LLN_PB}).
Second, we apply the method of lifting the strong law of large numbers for persistent Betti numbers to persistence diagrams, developed in~\cite{HST18}.

The principal aim in this article is to study the large deviation behavior of the persistence diagrams of random cubical filtrations.
In other words, we are interested in the small probability that the persistence diagram is far away from the deterministic limiting measure in the strong law of large numbers.
Though such an event is usually rare, they may exert a material impact on functionals of persistence diagrams.
The \textit{$($Donsker--Varadhan type$)$ large deviation principle} characterizes such small probabilities in terms of a rate function, which quantifies the rareness of the events (see Definition~\ref{df:LDP}).
We prove a large deviation principle for the persistence diagrams of a class of random cubical filtrations, and show that the rate function is given as the Fenchel--Legendre transform of the limiting logarithmic moment generating function (Theorem~\ref{thm:LDP_PD}).
For the proof, we establish a general method of lifting a large deviation principle for the tuples of persistent Betti numbers to persistence diagrams (Theorem~\ref{thm:LDP_PBtuple-PD}).
As described in Section~\ref{sec:LDP_PDs}, our method relies on the technique of exponentially good approximation in the large deviation theory.
Furthermore, in order to show the large deviation principle for the tuples of persistent Betti numbers, we generalize a large deviation principle for regular nearly additive real-valued processes, established by Sepp\"al\"ainen and Yukich~\cite{SY01}, to vector-valued processes with weaker nearly additive property (Theorem~\ref{thm:LDP_ernap}).

We remark on a few studies on the large deviation behaviors of several topological invariants.
Yogeshwaran and Adler~\cite{YSA17} proved a concentration inequality for the Betti numbers of a random geometric complex.
Samorodnitsky and Owada~\cite{SO22} studied an upper tail large deviation estimate for the Betti numbers of a multi-parameter random simplicial complex in the critical dimension.
Furthermore, Hirsch and Owada~\cite{HO22} proved a large deviation principle for counting measures associated to the configuration of a homogeneous Poisson point process in a Euclidean space.
As an application, they proved the large deviation principle for the first persistent Betti numbers of a filtration of random geometric complexes (alpha complexes) built over a homogeneous Poisson point process in $2$-dimensional Euclidean space.
In their setting, the large deviation principle for the higher-order persistent Betti numbers in higher-dimensional Euclidean spaces is still an open problem since they exploit the fact that the number of $1$-simplices in the planar Delaunay triangulation grows at most linearly with respect to the number of vertices, which is not the case for higher-dimensional simplex counts (see~\cite[Remark~4.2]{HO22} for details).
In contrast, in the case of random cubical filtration model, we prove large deviation principles for arbitrarily higher-order persistent Betti numbers (Theorem~\ref{thm:LDP_PBtuple}).
Moreover, Theorem~\ref{thm:LDP_PD} is the first result on the large deviation principle for persistence diagrams themselves throughout all the models.

The remainder of this article is organized as follows.
In Section~\ref{sec:mdl-rslt}, we provide some definitions, namely cubical set, cubical homology, cubical filtration, persistent homology, persistence diagram, and random cubical filtration.
Furthermore, we introduce our random cubical filtration model and summarize the main results.
Section~\ref{sec:LDP_ernap} presents a general large deviation principle for exponentially regular nearly additive vector-valued processes.
In Section~\ref{sec:LDP_PB}, we discuss the exponential regularity and exponentially near additivity of persistent Betti numbers.
In Section~\ref{sec:LDP_PDs}, we establish a general method of lifting a large deviation principle for the tuples of persistent Betti numbers to persistence diagrams using the technique of exponentially good approximation.
In Appendix~\ref{sec:LLN_rnap}, we prove the strong law of large numbers for strongly regular nearly additive vector-valued processes.
As we will see in Remark~\ref{rem:Exp-STG}, the strong regularity and strongly nearly additivity is weaker than the exponential regularity and exponentially nearly additivity, respectively.
In Appendix~\ref{sec:trans_LDP}, we review basic methods, needed in Section~\ref{sec:LDP_PDs}, to move around LDPs between different spaces.
In Appendix~\ref{sec:LDP_rdm_meas}, we provide a sufficient condition for an LDP for random measures, which is also used in Section~\ref{sec:LDP_PDs}.

\section{Model and main results}\label{sec:mdl-rslt}
Throughout this article, we fix $d\in\N$ as the dimension of the state space $\R^d$ where cubical sets and cubical filtrations are considered.
In Subsection~\ref{ssec:cube_hom}, we review the definitions of cubical sets and cubical homology concisely.
See Chapter~2 of~\cite{KMM04} for more detailed description.
In Subsection~\ref{ssec:PH_cube_filt}, we define the persistence diagram and persistent Betti number of cubical filtrations precisely.
In Subsection~\ref{ssec:RCF}, we introduce our random cubical filtration model.
Finally, we summarize our main results in Subsection~\ref{ssec:main_rslt}.

\subsection{Cubical homology}\label{ssec:cube_hom}
An \textit{elementary interval} is a closed interval $I\subset\R$ of the form $I=[l,l+1]$ or $I=\{l\}$ for some $l\in\Z$.
Such elementary intervals $I=[l,l+1]$ and $I=\{l\}$ are said to be \textit{nondegenerate} and \textit{degenerate}, respectively.
An \textit{elementary cube in $\R^d$} is a product set $I_1\times I_2\times\cdots\times I_d\subset\R^d$ of $d$ elementary intervals $I_1,I_2,\ldots,I_d$.
Let $\cK^d$ denote the set of all elementary cubes in $\R^d$.
Given an elementary cube $Q=I_1\times I_2\times\cdots\times I_d$ in $\R^d$, its \textit{dimension} $\dim Q$ is defined as the number of nondegenerate elementary intervals in $I_1,I_2,\ldots,I_d$.
We call an elementary cube $Q$ with $\dim Q=q$ an \textit{elementary $q$-cube in $\R^d$}.
For each $q\in\Z$, let $\cK^d_q$ be the set of all elementary $q$-cubes in $\R^d$.
A \textit{cubical set in $\R^d$} is a union of elementary cubes in $\R^d$.
Note here that the above union of elementary cubes is not necessarily a finite union unlike the definition in~\cite{KMM04}.
Instead, we call a finite union of elementary cubes in $\R^d$ a \textit{bounded} cubical set in $\R^d$.

Let $X$ be a cubical set in $\R^d$.
In what follows, we refer to an elementary cube $Q$ such that $Q\subset X$ as an \textit{elementary cube in $X$}.
For $q\in\Z$, let $\cK^d_q(X)$ denote the set of all elementary $q$-cubes in $X$.
The \textit{$q$th cubical chain group} $C_q(X)$ is defined as the $\R$-vector space consisting of all formal linear combinations of finitely many elementary $q$-cubes in $X$ with coefficients in $\R$.
Each element is called a \textit{cubical $q$-chain} and of the form $a_1Q_1+a_2Q_2+\cdots+a_mQ_m$
for some $a_i\in\R$ and $Q_i\in\cK^d_q$~($i=1,2,\ldots,m$).
Some authors use the notation $a_1\widehat Q_1+a_2\widehat Q_2+\cdots+a_m\widehat Q_m$ instead of $a_1Q_1+a_2Q_2+\cdots+a_mQ_m$ in order to stress that the elementary cubes are regarded as algebraic objects and that the linear combination here is in the formal sense.
Obviously, $C_q(X)=0$ for $q<0$ or $q>d$ since $\cK^d_q(X)=\emptyset$ in such cases.
Furthermore, $\cK^d_q(X)$ forms the canonical basis of $C_q(X)$ whenever $\cK^d_q(X)\neq\emptyset$.
For $q\in\Z$, we define the \textit{$q$th cubical boundary map} $\del_q^X\colon C_q(X)\to C_{q-1}(X)$ as the linear extension of
\[
\del_q^XQ
\coloneqq\sum_{j=1}^q(-1)^{j-1}(Q_j^+-Q_j^-)\in C_{q-1}(X)
\]
for any $Q=I_1\times I_2\times\cdots\times I_d\in C_q(X)$.
Here, $Q_j^+\in C_{q-1}(X)$ and $Q_j^-\in C_{q-1}(X)$ are defined by degenerating the $j$th nondegenerate elementary interval in $I_1,I_2,\ldots,I_d$ upward and downward, respectively.
More precisely, letting
\[
I_{i_1}=[l_1,l_1+1],I_{i_2}=[l_2,l_2+1],\ldots,I_{i_q}=[l_q,l_q+1]
\]
be the nondegenerate elementary intervals in $I_1,I_2,\ldots,I_d$, we define
\begin{align*}
Q_j^+&\coloneqq I_1\times\cdots\times I_{i_j-1}\times\{l_j+1\}\times I_{i_j+1}\times\cdots\times I_d
\shortintertext{and}
Q_j^-&\coloneqq I_1\times\cdots\times I_{i_j-1}\times\{l_j\}\times I_{i_j+1}\times\cdots\times I_d.
\end{align*}
\begin{eg}
Set $d=2$, and consider a cubical set $X=[0,1]^2$ in $\R^2$.
\begin{enumerate}
\item Let $Q=\{0\}\times\{0\}=\{(0,0)\}\in\cK^d_0(X)$.
Then,
\[
\del_0^XQ=0\in C_{-1}(X).
\]
\item Let $Q=[0,1]\times\{0\}\in\cK^d_1(X)$.
Then,
\[
\del_1Q
=\{1\}\times\{0\}-\{0\}\times\{0\}
=(1,0)-(0,0)\in C_0(X).
\]
\item Let $Q=[0,1]\times[0,1]\in\cK^d_2(X)$.
Then,
\begin{align*}
\del_2^XQ
&=(\{1\}\times[0,1]-\{0\}\times[0,1])-([0,1]\times\{1\}-[0,1]\times\{0\})\\
&=[0,1]\times\{0\}+\{1\}\times[0,1]-[0,1]\times\{1\}-\{0\}\times[0,1]\in C_1(X).
\end{align*}
\end{enumerate}
\end{eg}
For $q\in\Z$, define subspaces $Z_q(X)\coloneqq\ker\del_q^X$ and $B_q(X)\coloneqq\im\del_{q+1}^X$ of $C_q(X)$, which are called the \textit{$q$th cubical cycle group} and the \textit{$q$th cubical boundary group}, respectively.
A straightforward calculation yields $\del_q^X\circ\del_{q+1}^X=0$ for all $q\in\Z$, that is, $Z_q(X)\supset B_q(X)$.
The \textit{$q$th cubical homology group} $H_q(X)=H_q(X;\R)$ with coefficients in $\R$ is defined as the quotient $\R$-vector space $Z_q(X)/B_q(X)$.
When $X$ is a bounded cubical set in $\R^d$, the dimension of $H_q(X)$ is called the \textit{$q$th Betti number} of $X$, denoted by $\b_q(X)$.

\subsection{Persistent homology for cubical filtrations}\label{ssec:PH_cube_filt}
In this subsection, we review the definition of persistence diagram of an increasing family of cubical sets in $\R^d$.
A \textit{right-continuous cubical filtration in $\R^d$} is an increasing family $\X=\{X(t)\}_{t\ge0}$ of cubical sets $X(t)$ in $\R^d$ such that $X(t)=\bigcap_{t'>t}X(t')$ for every $t\ge0$.
In what follows, we omit the word ``right-continuous'', and simply call $\X$ a cubical filtration.
We say that a cubical filtration $\X=\{X(t)\}_{t\ge0}$ in $\R^d$ is \textit{bounded} if $\bigcup_{t\ge0}X(t)$ is bounded.
Note that if $\X=\{X(t)\}_{t\ge0}$ is a bounded cubical filtration in $\R^d$, then $X(t)$ differs from $\bigcup_{t'<t}X(t')$ only finitely many $t$'s.

Let $\R[\{z^t\colon t\ge0\}]$ be an $\R$-vector space of formal linear combinations of finitely many monomials $z^t$~($t\ge0$), where $z$ is an indeterminate.
The product of two elements in $\R[\{z^t\colon t\ge0\}]$ is defined by the linear extension of $az^t\cdot bz^{t'}\coloneqq abz^{t+t'}$~($a,b\in\R$, $t,t'\ge0$).
This operation equips $\R[\{z^t\colon t\ge0\}]$ with a graded ring structure.

Let $\X=\{X(t)\}_{t\ge0}$ be a bounded cubical filtration in $\R^d$.
For each $q\in\Z$, the \textit{$q$th persistent homology group} $H_q(\X)$ of $\X$ is defined by
\[
H_q(\X)\coloneqq\bigoplus_{t\ge0}H_q(X(t)).
\]
We define the action of monomial $z^u$~($u\ge0$) on $H_q(\X)$ by
\[
z^u\cdot(c_t+B_q(X(t)))_{t\ge0}\coloneqq(c'_t+B_q(X(t)))_{t\ge0}
\text{, where }
c'_t\coloneqq\begin{cases}
c_{t-u}	&\text{if $t\ge u$,}\\
0		&\text{if $t<u$.}
\end{cases}
\]
By the linear extension of the above action of monomials, $H_q(\X)$ has a graded module structure over the graded ring $\R[\{z^t\colon t\ge0\}]$.
The following theorem, which is often called the \textit{structure theorem of persistent homology group}, is crucial for defining the persistence diagram of $\X=\{X(t)\}_{t\ge0}$.
\begin{thm}[{\cite[Theorem~2.1]{ZC05}}]
Let $q\in\Z$ be fixed.
There exists a finite family $\{(b_i,d_i)\}_{i=1}^p$ with $0\le b_i<d_i\le\infty$ such that the following graded module isomorphism holds$:$
\[
H_q(\X)\simeq\bigoplus_{i=1}^p((z^{b_i})/(z^{d_i})).
\]
Here, $(z^t)$ expresses an ideal in $\R[\{z^t\colon t\ge0\}]$ generated by the monomial $z^t$, and $(z^\infty)$ is regarded as the zero ideal.
Furthermore, $\{(b_i,d_i)\}_{i=1}^p$ is uniquely determined as a multiset.
\end{thm}
In the above theorem, $\{b_i\}_{i=1}^p$ and $\{d_i\}_{i=1}^p$ are called the \textit{$q$th birth times} and \textit{death times}, respectively, and each pair $(b_i,d_i)$ is called the \textit{$q$th birth-death pair} of the cubical filtration $\X=\{X(t)\}_{t\ge0}$.
Intuitively speaking, each birth-death pair $(b_i,d_i)$ corresponds to a $q$-dimensional hole that appears at time $b_i$, persists over the time-interval $[b_i,d_i)$, and disappears at time $d_i$.
We note that the number of the $q$th birth-death pair of the cubical filtration $\X=\{X(t)\}_{t\ge0}$ is trivially bounded above by the number of elementary $q$-cubes in $\bigcup_{t\ge0}X(t)$.

Now, let us write
\[
\Dl=\{(s,t)\in[0,\infty]^2\mid 0\le s<t\le\infty\},
\]
which is naturally homeomorphic to $\{(x,y)\mid0\le x<y\le1\}$ equipped with the usual topology.
For $q\in\Z$, we define the \textit{$q$th persistence diagram} $\xi_q^{\X}$ of the cubical filtration $\X=\{X(t)\}_{t\ge0}$ as a counting measure
\[
\xi_q^{\X}\coloneqq\sum_{i=1}^p\dl_{(b_i,d_i)}
\]
on $\Dl$, where $\dl_{(b_i,d_i)}$ is the Dirac measure at $(b_i,d_i)$, i.e., for any Borel set $A\subset\Dl$,
\[
\dl_{(b_i,d_i)}(A)\coloneqq\begin{cases}
1	&\text{if $(b_i,d_i)\in A$,}\\
0	&\text{if $(b_i,d_i)\notin A$.}
\end{cases}
\]
In order to deal with the convergence of persistence diagrams, we will regard each persistence diagram as an element of the space $\cM(\Dl)$ of Radon measures on $\Dl$ defined below.
A Borel measure $\xi$ on $\Dl$ is called a \textit{Radon measure} if $\xi(K)<\infty$ for any compact set $K\subset\Dl$.
Let $\cM(\Dl)$ denote the set of all Radon measures on $\Dl$.
We equip $\cM(\Dl)$ with the vague topology, i.e., the weakest topology such that for any $f\in C_c(\Dl)$, the map $\cM(\Dl)\ni\xi\mapsto\int_\Dl f\,d\xi\in\R$ is continuous.
Here, $C_c(\Dl)$ be the set of all real-valued continuous functions on $\Dl$ with compact support.
Note that for a sequence $\{\xi^n\}_{n\in\N}$ in $\cM(\Dl)$ and $\xi\in \cM(\Dl)$, the Radon measure $\xi^n$ converges vaguely to $\xi$ as $n\to\infty$ if and only if $\lim_{n\to\infty}\int_\Dl f\,d\xi^n=\int_\Dl f\,d\xi$ for any $f\in C_c(\Dl)$.

Next, we review the notion of persistent Betti number.
Let $0\le s\le t<\infty$.
We denote by $\iota_s^t$ the inclusion map from $X(s)$ to $X(t)$, and by $(\iota_s^t)_*\colon H_q(X(s))\to H_q(X(t))$ the induced linear map of $\iota_s^t$.
We call the rank of the map $(\iota_s^t)_*$ the \textit{$q$th persistent Betti number of $\X$ at $(s,t)$}, and denote it by $\b_q^\X(s,t)$.
The notion of persistent Betti number is a generalization of Betti number.
Indeed, $\b_q^\X(t,t)=\b_q(X(t))$ holds for every $t\ge0$.
Since
\[
\im(\iota_s^t)_*
\simeq\frac{H_q(X(s))}{\ker(\iota_s^t)_*}
=\frac{Z_q(X(s))/B_q(X(s))}{(Z_q(X(s))\cap B_q(X(t)))/B_q(X(s))}
\simeq\frac{Z_q(X(s))}{Z_q(X(s))\cap B_q(X(t))},
\]
we have
\begin{equation}\label{eq:PB}
\b_q^\X(s,t)=\dim\frac{Z_q(X(s))}{Z_q(X(s))\cap B_q(X(t))}.
\end{equation}
Intuitively speaking, $\b_q^\X(s,t)$ expresses the number of $q$-dimensional holds that appear before time $s$ and persist to time $t$ in the filtration $\X=\{X(t)\}_{t\ge0}$.
The following relationship between the persistence diagram and the persistent Betti number is highly important, which is called the $k$-triangle lemma in~\cite{CEH07,ELZ02}.
\begin{thm}[\cite{CEH07,ELZ02}]\label{thm:mass_quadrant}
Let $q\in\Z$ and $0\le s\le t<\infty$ be fixed.
Then, it holds that
\begin{equation}\label{eq:mass_quadrant}
\xi_q^\X([0,s]\times(t,\infty])=\b_q^\X(s,t).
\end{equation}
\end{thm}
\begin{rem}
The persistence diagram $\xi_q^\X$ is in fact characterized as the unique counting measure on $\Dl$ satisfying~\eqref{eq:mass_quadrant} for any $0\le s\le t<\infty$.
\end{rem}
The following is an immediate corollary of Theorem~\ref{thm:mass_quadrant} together with the inclusion-exclusion principle.
\begin{cor}
Let $q\in\Z$ and $0\le s_1\le s_2\le t_1\le t_2<\infty$ be fixed.
Then,
\begin{align*}
\xi_q^\X([0,s_2]\times(t_1,t_2])
&=\b_q^\X(s_2,t_1)-\b_q^\X(s_2,t_2)
\shortintertext{and}
\xi_q^\X((s_1,s_2]\times(t_1,t_2])
&=\b_q^\X(s_2,t_1)-\b_q^\X(s_2,t_2)+\b_q^\X(s_1,t_2)-\b_q^\X(s_1,t_1).
\end{align*}
\end{cor}

\subsection{Random cubical filtration model}\label{ssec:RCF}
For a cubical filtration $\X=\{X(t)\}_{t\ge0}$ in $\R^d$ and an elementary cube $Q\in\cK^d$, the \textit{birth time of $Q$ in $\X$} is defined by
\[
t^\X_Q\coloneqq\inf\{t\ge0\mid Q\in X(t)\}.
\]
By convention, we regard $t^\X_Q\coloneqq\infty$ if $Q\notin\bigcup_{t\ge0}X(t)$.
Obviously, $Q'\subset Q\in\cK^d$ implies $t^\X_{Q'}\le t^\X_Q$.
Conversely, given a family $\{t_Q\}_{Q\in\cK^d}$ in $[0,\infty]$ satisfying that
\begin{equation}\label{eq:birth_condi}
Q'\subset Q\in\cK^d\Rightarrow t_{Q'}\le t_Q,
\end{equation}
we can define a cubical filtration $\X=\{X(t)\}_{t\ge0}$ in $\R^d$ so that $t^\X_Q=t_Q$ for any $Q\in\cK^d$.
Indeed, we may simply set
\[
X(t)=\bigcup\{Q\in\cK^d\mid t_Q\le t\}
\]
for every $t\ge0$.
We call such $\X=\{X(t)\}_{t\ge0}$ the \textit{cubical filtration in $\R^d$ corresponding to} $\{t_Q\}_{Q\in\cK^d}$.

Now, let $\cC^d$ be the set of all cubical filtrations in $\R^d$, and let $\cF^d$ be the smallest $\sg$-field such that the map $\cC^d\ni\X\mapsto t^\X_Q\in[0,\infty]$ is measurable for any $Q\in\cK^d$.
In other words, $\cF^d$ is the $\sg$-field generated by the maps $\{\cC^d\ni\X\mapsto t^\X_Q\in[0,\infty]\colon Q\in\cK^d\}$.
We call a random variable taking values in the measurable space $(\cC^d,\cF^d)$ a \textit{random cubical filtration in $\R^d$}.

Next, we introduce our random cubical filtration model.
For any subsets $A,B\subset\R^d$, define $d_{\max}(A,B)\coloneqq\inf\{\|x-y\|_{\max}\mid x\in A,y\in B\}$, where $\|\cdot\|_{\max}$ is the max norm in $\R^d$.
In this article, we consider a random cubical filtration $\X=\{X(t)\}_{t\ge0}$ in $\R^d$ satisfying the following two assumptions.
\begin{asm}[Stationarity]\label{asm:stationarity}
For every $z\in\Z^d$, the $[0,\infty]^{\cK^d}$-valued random variables $\{t^\X_Q\}_{Q\in\cK^d}$ and $\{t^\X_{z+Q}\}_{Q\in\cK^d}$ have the same probability distribution.
Here, $z+Q\coloneqq\{z+x\mid x\in Q\}\in\cK^d$ for any $z\in\Z^d$ and $Q\in\cK^d$.
In such case, we say that $\X$ is \textit{stationary}.
\end{asm}
\begin{asm}[Local dependence]\label{asm:loc_dep}
There exists an integer $R\ge0$ such that for any subsets $A,B\subset\R^d$ with $d_{\max}(A,B)>R$, the families $\{t^\X_Q\colon Q\in\cK^d,Q\subset A\}$ and $\{t^\X_Q\colon Q\in\cK^d,Q\subset B\}$ are independent.
In such case, we say that $\X$ is \textit{$R$-dependent}.
\end{asm}

As typical random cubical filtration models that satisfy Assumptions~\ref{asm:stationarity} and~\ref{asm:loc_dep}, we introduce the upper and lower random cubical filtrations.
\begin{eg}\label{eg:upper_lower}
Let $\{F_q\}_{q=0}^d$ be a family of probability distribution functions on $[0,\infty]$, i.e., $F_q$ is a right-continuous function on $[0,\infty]$ with $F_q(\infty)=1$ (while not necessarily $F_q(0)=0$ or $\lim_{x\to\infty}F_q(x)=1$) for each $0\le q\le d$.
To each elementary cube $Q\in\cK^d$, we assign a $[0,\infty]$-valued random variable $u_Q$ with probability distribution function $F_{\dim Q}$ independently.
For each $Q\in\cK^d$, we set
\[
\overline t_Q=\min\{u_{Q'}\mid Q'\in\cK^d,Q'\supset Q\}
\quad\text{and}\quad
\underline t_Q=\max\{u_{Q'}\mid Q'\in\cK^d,Q'\subset Q\}.
\]
Noting that both the families $\{\overline t_Q\}_{Q\in\cK^d}$ and $\{\underline t_Q\}_{Q\in\cK^d}$ satisfy~\eqref{eq:birth_condi}, we define $\overline\X=\{\overline X(t)\}_{t\ge0}$ and $\underline\X=\{\underline X(t)\}_{t\ge0}$ as the random cubical filtrations in $\R^d$ corresponding to $\{\overline t_Q\}_{Q\in\cK^d}$ and $\{\underline t_Q\}_{Q\in\cK^d}$, respectively.
Obviously, $\overline\X$ and $\underline\X$ are stationary, also $1$- and $0$-dependent, respectively.
We call $\overline\X$ and $\underline\X$ the \textit{upper} and \textit{lower random cubical filtrations}, respectively, with probability distribution functions $\{F_q\}_{q=0}^d$.
\end{eg}
\begin{rem}
The word ``upper'' and ``lower'' in Examples~\ref{eg:upper_lower} derives from the upper and lower random simplicial complex model, extensively studied in~\cite{CF16a,CF17a,CF17b,CF16b,FMN22,Fo19,HK19}.
In fact, for every $t\ge0$, the random cubical sets $\overline X(t)$ and $\underline X(t)$ can be regarded as the cubical versions of the upper and lower random simplicial complex with parameters $\{F_q(t)\}_{q=0}^d$, respectively.
\end{rem}
We additionally introduce other random cubical filtration models that satisfy Assumptions~\ref{asm:stationarity} and~\ref{asm:loc_dep}, where the birth times of elementary cubes are given in more geometric ways.
For lattice points $z,z'\in\Z^d$, we say that $z$ and $z'$ are \textit{adjacent} if $\|z-z'\|_{L^1}=1$.
Here, $\|\cdot\|_{L^1}$ is the $L^1$-norm in $\R^d$.
\begin{eg}\label{eg:plattice_RCF}
Let $\mu$ be a probability measure on $\R^d$.
Let $\{\eps_z\}_{z\in\Z^d}$ be i.i.d. random variables drawn from $\mu$, and define $x_z\coloneqq z+\eps_z$ for every $z\in\Z^d$.
For each $Q\in\cK^d$, we set
\[
t_Q=\inf\{t\ge0\mid\text{$\|x_z-x_{z'}\|_{\R^d}\le t$ for any adjacent lattice points $z$ and $z'$ in $Q$}\}.
\]
Noting that the family $\{t_Q\}_{Q\in\cK^d}$ satisfies~\eqref{eq:birth_condi}, we define $\X=\{X(t)\}_{t\ge0}$ as the random cubical filtration in $\R^d$ corresponding to $\{t_Q\}_{Q\in\cK^d}$.
Obviously, $\X$ is stationary and $0$-dependent.
\end{eg}
\begin{eg}
Let $\mu$ be a probability measure on $\R^d$ with compact support.
Let $\{\eps_z\}_{z\in\Z^d}$ and $\{x_z\}_{z\in\Z^d}$ be the same as in Example~\ref{eg:plattice_RCF}.
For each $Q\in\cK^d$, we set
\[
t_Q=\inf\left\{t\ge0\relmiddle|Q\subset\bigcup_{z\in\Z^d}\bar B(x_z,t)\right\}.
\]
Here, $\bar B(x_z,t)$ is the closed ball of radius $t$ centered at $x_z$.
Noting again that the family $\{t_Q\}_{Q\in\cK^d}$ satisfies~\eqref{eq:birth_condi}, we define $\X=\{X(t)\}_{t\ge0}$ as the random cubical filtration in $\R^d$ corresponding to $\{t_Q\}_{Q\in\cK^d}$.
Obviously, $\X$ is stationary.
Furthermore, $\X$ is locally dependent since $\mu$ has a compact support.
\end{eg}

\subsection{Main results}\label{ssec:main_rslt}
In this subsection, we summarize our main results.
For each $n\in\N$, we set a rectangular region
\[
\Lm^n=[-n,n]^d\subset\R^d.
\]
Given $n\in\N$ and a random cubical filtration $\X=\{X(t)\}_{t\ge0}$ in $\R^d$, define a restricted random cubical filtration $\X^n=\{X^n(t)\}_{t\ge0}$ by
\[
X^n(t)\coloneqq X(t)\cap\Lm^n
\]
for every $t\ge0$.
Note that $\X^n=\{X^n(t)\}_{t\ge0}$ is a random bounded cubical filtration.
In what follows, $|A|$ denotes the $d$-dimensional Lebesgue measure of a Borel subset $A\subset\R^d$.
In particular, $|\Lm^n|=(2n)^d$.

The first result is the strong law of large numbers for the persistent Betti numbers of a random cubical filtration satisfying Assumptions~\ref{asm:stationarity} and~\ref{asm:loc_dep}.
\begin{thm}\label{thm:LLN_PB}
Let $\X=\{X(t)\}_{t\ge0}$ be a random cubical filtration in $\R^d$ satisfying Assumptions~\ref{asm:stationarity} and~\ref{asm:loc_dep}.
Fix an integer $0\le q<d$ and $0\le s\le t<\infty$.
Then, there exists a constant $\widehat\b_q(s,t)\in[0,\infty)$, depending on $q$, $s$, and $t$, such that
\[
\frac{\E[\b_q^{\X^n}(s,t)]}{|\Lm^n|}\to\widehat\b_q(s,t)\quad\text{as $n\to\infty$}
\]
and
\[
\frac{\b_q^{\X^n}(s,t)}{|\Lm^n|}\to\widehat\b_q(s,t)\quad\text{almost surely as $n\to\infty$.}
\]
\end{thm}
\begin{rem}
Since $\b_q^{\X^n}(t,t)=\b_q(X^n(t))$ for every $t\ge0$, Theorem~\ref{thm:LLN_PB} implies the strong law of large number for Betti numbers.
This was first obtained by Hiraoka and Tsunoda~\cite[Theorem~2.8]{HT18} for the upper random cubical filtration model with a slightly general setting, where the probability distribution of $\{u_Q\}_{Q\in\cK^d}$ taken in Example~\ref{eg:upper_lower} is ergodic with the canonical translation on $\Z^d$.
See~\cite[Section~2]{HT18} for more details.
\end{rem}

The next result is the strong law of large numbers for the persistence diagrams of a random cubical filtration satisfying Assumptions~\ref{asm:stationarity} and~\ref{asm:loc_dep}.
For every $n\in\N$ and $q\in\Z$, the mean measure $\E[\xi_q^{\X^n}]$ of $\xi_q^{\X^n}$ is defined by $\E[\xi_q^{\X^n}](A)\coloneqq\E[\xi_q^{\X^n}(A)]$ for any Borel set $A\subset\Dl$.
Since the number of the $q$th birth-death pair of $\X^n$ is bounded above by the number of elementary $q$-cubes in $\Lm^n$, we have
\begin{equation}\label{eq:total_mass}
\xi_q^{\X^n}(\Dl)\le\#\cK^d_q(\Lm^n).
\end{equation}
In particular, $\E[\xi_q^{\X^n}]$ is a Radon measure on $\Dl$.
\begin{thm}\label{thm:LLN_PD}
Let $\X=\{X(t)\}_{t\ge0}$ be a random cubical filtration in $\R^d$ satisfying Assumptions~\ref{asm:stationarity} and~\ref{asm:loc_dep}.
Fix an integer $0\le q<d$.
Then, there exists a Radon measure $\widehat\xi_q$ on $\Dl$ such that
\[
\frac{\E[\xi_q^{\X^n}]}{|\Lm^n|}\xrightarrow[]{\text{v}}\widehat\xi_q\quad\text{as $n\to\infty$.}
\]
Here, $\xrightarrow[]{\text{v}}$ denotes the vague convergence of Radon measures on $\Dl$.
Furthermore,
\[
\frac{\xi_q^{\X^n}}{|\Lm^n|}\xrightarrow[]{\text{v}}\widehat\xi_q\quad\text{almost surely as $n\to\infty$.}
\]
\end{thm}

The principal aim in this article is to investigate large deviation principles for the persistence diagrams of a random cubical filtration satisfying Assumptions~\ref{asm:stationarity} and~\ref{asm:loc_dep}.
We start with the definition of large deviation principle in a general setting.
\begin{df}\label{df:LDP}
Let $\{a_n\}_{n\in\N}$ be a sequence of positive numbers tending to infinity as $n\to\infty$.
Let $\cX$ be a Hausdorff topological space equipped with the Borel $\sg$-algebra.
An $\cX$-valued process $\{S^n\}_{n\in\N}$, i.e., a sequence of $\cX$-valued random variables, satisfies a \textit{$($Donsker--Varadhan type$)$ large deviation principle $($LDP$)$} with speed $a_n$ if there exists a lower semicontinuous function $I\colon\cX\to[0,\infty]$ such that
\begin{itemize}
\item for any closed set $F\subset\cX$,
\begin{equation}\label{eq:upper_bound}
\limsup_{n\to\infty}a_n^{-1}\log\P(S^n\in F)\le-\inf_{x\in F}I(x),
\end{equation}
\item for any open set $G\subset\cX$,
\[
\liminf_{n\to\infty}a_n^{-1}\log\P(S^n\in G)\ge-\inf_{x\in G}I(x).
\]
\end{itemize}
The function $I$ is called a \textit{rate function}.
If, furthermore, the sublevel set $\{x\in\cX\mid I(x)\le\a\}$ is compact for every $\a\in[0,\infty)$, then $I$ is called a \textit{good rate function}.
\end{df}
\begin{rem}
By taking $F=\cX$ in~\eqref{eq:upper_bound}, we have $\inf_{x\in\cX}I(x)=0$.
When $I$ is a good rate function, this implies that $I$ has at least one (not necessarily unique) zero point.
\end{rem}
The next result is the LDP for the tuples of the persistent Betti numbers of a random cubical filtration satisfying Assumptions~\ref{asm:stationarity} and~\ref{asm:loc_dep}.
Before stating the result, we provide the basics of the Fenchel--Legendre transform.
In what follows, $\langle\cdot,\cdot\rangle_{\R^h}$ denotes the canonical inner product in $\R^h$, and $\|\cdot\|_{\R^h}$ is its induced norm.
Given a function $\varphi\colon\R^h\to[-\infty,\infty]$, its \textit{Fenchel--Legendre transform} $\varphi^*\colon\R^h\to[-\infty,\infty]$ is defined by
\begin{equation}\label{eq:FL_transform}
\varphi^*(x)\coloneqq\sup_{\lm\in\R^h}\{\langle\lm,x\rangle_{\R^h}-\varphi(\lm)\}
\end{equation}
for any $x\in\R^h$.
Every Fenchel--Legendre transform is convex and lower semicontinuous since it is the supremum of affine functions.
If $\varphi(0)=0$, then we can check that $\varphi^*(x)\in[0,\infty]$ for every $x\in\R^h$ by taking $\lm=0$ in~\eqref{eq:FL_transform}.
In the large deviation theory, the Fenchel--Legendre transform appears as a natural candidate for
rate functions of LDPs in a general setting~(see, e.g., Theorem~4.5.3~(b) in~\cite{DZ10}).
We now state our first LDP result.
\begin{thm}\label{thm:LDP_PBtuple}
Let $\X=\{X(t)\}_{t\ge0}$ be a random cubical filtration in $\R^d$ satisfying Assumptions~\ref{asm:stationarity} and~\ref{asm:loc_dep}.
Fix an integer $0\le q<d$ and a finite family $\cP=\{(s_i,t_i)\}_{i=1}^h$ with $0\le s_i\le t_i<\infty$.
Then, for every $\lm=(\lm_1,\ldots,\lm_h)\in\R^h$, the limit
\[
\varphi_{q,\cP}(\lm)
\coloneqq\lim_{n\to\infty}|\Lm^n|^{-1}\log\E\biggl[\exp\biggl(\sum_{i=1}^h\lm_i\b_q^{\X^n}(s_i,t_i)\biggr)\biggr]
\]
exists in $\R$, and the $\R^h$-valued process
\[
\biggl\{\biggl(\frac{\b_q^{\X^n}(s_1,t_1)}{|\Lm^n|},\frac{\b_q^{\X^n}(s_2,t_2)}{|\Lm^n|},\ldots,\frac{\b_q^{\X^n}(s_h,t_h)}{|\Lm^n|}\biggr)\biggr\}_{n\in\N}
\]
satisfies the LDP with speed $|\Lm^n|$ and good convex rate function $\varphi_{q,\cP}^*\colon\R^h\to[0,\infty]$.
Furthermore, $\varphi_{q,\cP}^*(x)=0$ if and only if $x=(\widehat\b_q(s_1,t_1),\ldots,\widehat\b_q(s_h,t_h))$.
\end{thm}

Next, we state the LDP for the persistence diagrams of a random cubical filtration satisfying Assumptions~\ref{asm:stationarity} and~\ref{asm:loc_dep}.
\begin{thm}\label{thm:LDP_PD}
Let $\X=\{X(t)\}_{t\ge0}$ be a random cubical filtration in $\R^d$ satisfying Assumptions~\ref{asm:stationarity} and~\ref{asm:loc_dep}.
Fix an integer $0\le q<d$.
Then, for every $f\in C_c(\Dl)$, the limit
\[
\varphi_q(f)\coloneqq\lim_{n\to\infty}|\Lm^n|^{-1}\log\E\biggl[\exp\biggl(\int_\Dl f\,d\xi_q^{\X^n}\biggr)\biggr]
\]
exists in $\R$, and the $\cM(\Dl)$-valued process $\{\xi_q^{\X^n}/|\Lm^n|\}_{n\in\N}$ satisfies the LDP with speed $|\Lm^n|$ and a good convex rate function $I_q\colon \cM(\Dl)\to[0,\infty]$ defined by
\begin{equation}\label{eq:LDP_PD}
I_q(\xi)\coloneqq\sup_{f\in C_c(\Dl)}\biggl\{\int_\Dl f\,d\xi-\varphi_q(f)\biggr\}
\end{equation}
for any $\xi\in\cM(\Dl)$.
Furthermore, $I_q(\xi)=0$ if and only if $\xi=\widehat\xi_q$.
\end{thm}

\section{Large deviation principle for exponentially regular nearly additive processes}\label{sec:LDP_ernap}
In this section, we develop a general LDP result for exponentially regular nearly additive vector-valued processes, which is crucial for the proof of Theorem~\ref{thm:LDP_PBtuple} in Section~\ref{sec:LDP_PB}.
In Subsection~\ref{ssec:stmt_rslt}, we define the notions of exponentially near additivity and exponential regularity, and state the general LDP result (Theorem~\ref{thm:LDP_ernap}).
Subsection~\ref{ssec:pf_LDP_ernap} presents the proof of Theorem~\ref{thm:LDP_ernap}.
The proofs of technical lemmas needed in the proof of Theorem~\ref{thm:LDP_ernap} is deferred to Subsection~\ref{ssec:pf_3lemmas}.

\subsection{Statement of result}\label{ssec:stmt_rslt}
Throughout this section, we fix $h\in\N$, and consider an $\R^h$-valued process $\{S^n\}_{n\in\N}$, i.e., a sequence of $\R^h$-valued random variables.
In the applications dealt with in Section~\ref{sec:LDP_PB}, $S^n$ is taken to be a random vector associated to the rectangular region $\Lm^n=[-n,n]^d$, and $|\Lm^n|$ indicates its $d$-dimensional Lebesgue measure.
However, we here regard $S^n$ and $|\Lm^n|\coloneqq(2n)^d$ as just a random vector indexed by $n$ and a scaling factor, respectively.

The following notions are crucial for stating the main theorem in this section (Theorem~\ref{thm:LDP_ernap}).
\begin{df}\label{df:enap}
Let $r\ge0$ be an integer.
We say that an $\R^h$-valued process $\{S^n\}_{n\in\N}$ is \textit{exponentially $r$-nearly additive} if there exist $\R^h$-valued random variables $\{S^{n,z}\}_{n\in\N,z\in\Z^d}$ such that the following conditions are satisfied:
\begin{itemize}
\item $\{S^{n,z}\}_{z\in\Z^d}$ are independent copies of $S^n$ for every $n\in\N$; 
\item for any $\eps>0$ and $C>0$, there exists an integer $K>r$ such that
\[
\P\biggl(\biggl\|S^{(2m+1)k}-\sum_{z\in\Z^d\cap[-m,m]^d}S^{k-r,z}\biggr\|_{\R^h}>\eps|\Lm^{(2m+1)k}|\biggr)\le\exp(-C|\Lm^{(2m+1)k}|)
\]
for all $k\ge K$ and $m\in\N$.
\end{itemize}
We also say that an $\R^h$-valued process $\{S^n\}_{n\in\N}$ is \textit{exponentially nearly additive} if there exists an integer $r\ge0$ such that $\{S^n\}_{n\in\N}$ is exponentially $r$-nearly additive.
\end{df}
\begin{rem}
The exponentially near additivity in Definition~\ref{df:enap} with $h=1$ and $r=0$ corresponds to the definition of the near additivity in~Assumption~2.1 of~\cite{SY01}.
\end{rem}
\begin{df}
We say that an $\R^h$-valued process $\{S^n\}_{n\in\N}$ is \textit{exponentially regular} if the following property holds for each fixed $k\in\N$:
if $m_n$ is taken as the unique integer satisfying that $(2m_n+1)k\le n<(2m_n+3)k$ for each $n\in\N$, then for any $\eps>0$ and $C>0$, there exists $N\in\N$ such that
\[
\P(\|S^n-S^{(2m_n+1)k}\|_{\R^h}>\eps|\Lm^n|)\le\exp(-C|\Lm^n|)
\]
for all $n\ge N$.
\end{df}

Let $\{S^n\}_{n\in\N}$ be an exponentially regular nearly additive $\R^h$-valued process consisting of integrable random variables.
If $\sup_{n\in\N}\|\E[S^n]\|_{\R^h}/|\Lm^n|<\infty$, then $\{S^n\}_{n\in\N}$ satisfies a strong law of large numbers, i.e., the limit
\begin{equation}\label{eq:lim_ESn}
\widehat S\coloneqq\lim_{n\to\infty}\frac{\E[S^n]}{|\Lm^n|}
\end{equation}
exists in $\R^h$, and
\[
\frac{S^n}{|\Lm^n|}\to\widehat S\quad\text{almost surely as $n\to\infty$.}
\]
See Appendix~\ref{sec:LLN_rnap} for the proof under a weaker assumption.

The following is a large deviation principle for exponentially regular nearly additive processes.
\begin{thm}\label{thm:LDP_ernap}
Let $\{S^n\}_{n\in\N}$ be an exponentially regular nearly additive $\R^h$-valued process satisfying that
\begin{equation}\label{eq:LDP_ernap_0}
\sup_{n\in\N}|\Lm^n|^{-1}\log\E[\exp(\langle\lm,S^n\rangle_{\R^h})]<\infty
\end{equation}
for any $\lm\in\R^h$.
Then, for every $\lm\in\R^h$, the limit
\[
\varphi(\lm)\coloneqq\lim_{n\to\infty}|\Lm^n|^{-1}\log\E[\exp(\langle\lm,S^n\rangle_{\R^h})]
\]
exists in $\R$, and the $\R^h$-valued process $\{S^n/|\Lm^n|\}_{n\in\N}$ satisfies the LDP with speed $|\Lm^n|$ and good convex rate function $\varphi^*\colon\R^h\to[0,\infty]$.
Furthermore, $\varphi^*(x)=0$ if and only if $x=\widehat S$, defined by~\eqref{eq:lim_ESn}.
\end{thm}
\begin{rem}
In the above theorem, we note that if~\eqref{eq:LDP_ernap_0} holds for any $\lm\in\R^h$, then each $S^n$ is integrable and $\sup_{n\in\N}\|\E[S^n]\|_{\R^h}/|\Lm^n|<\infty$ holds from an elementary calculation.
Hence, the limit $\widehat S$ defined in~\eqref{eq:lim_ESn} exists.
\end{rem}
\begin{rem}
The above theorem can be regarded as a generalization of Theorem~2.1 in~\cite{SY01}, where exponentially regular $0$-nearly additive real-valued processes are considered.
\end{rem}
Combining Theorem~\ref{thm:LDP_ernap} and the preceding discussion on the strong law of large numbers, we immediately obtain the following useful corollary.
\begin{cor}\label{cor:LDP_ernap}
Let $\{S^n\}_{n\in\N}$ be an exponentially regular nearly additive $\R^h$-valued process satisfying that
\[
\sup_{n\in\N}|\Lm^n|^{-1}\log\E[\exp(\langle\lm,S^n\rangle_{\R^h})]<\infty
\]
for any $\lm\in\R^h$.
Then, the following statements hold.
\begin{enumerate}
\item The limit
\[
\widehat S\coloneqq\lim_{n\to\infty}\frac{\E[S^n]}{|\Lm^n|}
\]
exists in $\R^h$, and
\[
\frac{S^n}{|\Lm^n|}\to\widehat S\quad\text{almost surely as $n\to\infty$.}
\]
\item For every $\lm\in\R^h$, the limit
\[
\varphi(\lm)\coloneqq\lim_{n\to\infty}|\Lm^n|^{-1}\log\E[\exp(\langle\lm,S^n\rangle_{\R^h})]
\]
exists in $\R$, and the $\R^h$-valued process $\{S^n/|\Lm^n|\}_{n\in\N}$ satisfies the LDP with speed $|\Lm^n|$ and good convex rate function $\varphi^*\colon\R^h\to[0,\infty]$.
Furthermore, $\varphi^*(x)=0$ if and only if $x=\widehat S$.
\end{enumerate}
\end{cor}

\subsection{Proof of Theorem~\ref{thm:LDP_ernap}}\label{ssec:pf_LDP_ernap}
The proof of Theorem~\ref{thm:LDP_ernap} relies mainly on two theorems in the large deviation theory.
The first one shows the existence of an LDP for the $\R^h$-valued process $\{S^n/|\Lm^n|\}_{n\in\N}$ with a (not necessarily convex) rate function $I\colon\R^h\to[0,\infty]$ (Theorem~\ref{thm:DZ_4.4.10}).
The second one guarantees that if the rate function $I$ is convex, then $I$ is given as the Fenchel--Legendre transform of the limiting logarithmic moment generating function (Theorem~\ref{thm:DZ_4.5.10}).

For these theorems, we first review basic notions.
An $\R^h$-valued process $\{S^n\}_{n\in\N}$ is said to be \textit{exponentially tight} (with speed $|\Lm^n|$) if for any $\a>0$, there exists a compact set $K\subset\R^h$ such that
\[
\limsup_{n\to\infty}|\Lm^n|^{-1}\log\P(S^n\notin K)\le-\a.
\]
A function $I\colon\R^h\to[0,\infty]$ is said to be \textit{convex} if for any $x_1,x_2\in\R^h$ and $t\in(0,1)$, it holds that
\begin{equation}\label{eq:convex}
tI(x_1)+(1-t)I(x_2)\ge I(tx_1+(1-t)x_2).
\end{equation}
In what follows in this section, we use the following notation: for any Borel function $f\colon\R^h\to\R$ and $n\in\N$,
\[
\Gm_n(f)\coloneqq\log\E\biggl[\exp\biggl(|\Lm^n|f\biggl(\frac{S^n}{|\Lm^n|}\biggr)\biggr)\biggr]\in(-\infty,\infty].
\]
Furthermore, let $\sF(\R^h)$ be the class of Lipschitz continuous and concave real-valued functions on $\R^h$.
Here, a real-valued function $f$ on $\R^h$ is said to be \textit{concave} if $-f$ satisfies~\eqref{eq:convex} for any $x_1,x_2\in\R^h$ and $t\in(0,1)$.
The class $\sF(\R^h)$ is well-separating in the sense that
\begin{itemize}
\item $\sF(\R^h)$ contains the constant functions,
\item $\sF(\R^h)$ is closed under finite pointwise minima, i.e., $f_1,f_2\in\sF(\R^h)$ implies $f_1\wedge f_2\in\sF(\R^h)$,
\item $\sF(\R^h)$ separates points in $\R^h$, i.e., for any two points $x\neq y$ in $\R^h$ and $a,b\in\R$, there exists a function $f\in\sF(\R^h)$ such that both $f(x)=a$ and $f(y)=b$ hold.
\end{itemize}
The following theorem is a special case of Theorem 4.4.10 in~\cite{DZ10} with the state space $\R^h$ and the well-separating class $\sF(\R^h)$.
\begin{thm}[{\cite[Theorem 4.4.10]{DZ10}}]\label{thm:DZ_4.4.10}
Let $\{S^n\}_{n\in\N}$ be an $\R^h$-valued process.
Suppose that the $\R^h$-valued process $\{S^n/|\Lm^n|\}_{n\in\N}$ is exponentially tight, and that the limit $\lim_{n\to\infty}|\Lm^n|^{-1}\Gm_n(f)$ exists in $[-\infty,\infty)$ for any $f\in\sF(\R^h)$.
Then, for every $f\in C_b(\R^h)$, the limit
\[
\Gm(f)\coloneqq\lim_{n\to\infty}|\Lm^n|^{-1}\Gm_n(f)
\]
also exists in $\R$, and the $\R^h$-valued process $\{S^n/|\Lm^n|\}_{n\in\N}$ satisfies the LDP with speed $|\Lm^n|$ and a good rate function $I\colon\R^h\to[0,\infty]$ defined by
\begin{equation}\label{eq:Cb-type_FL}
I(x)\coloneqq\sup_{f\in C_b(\R^h)}\{f(x)-\Gm(f)\}.
\end{equation}
\end{thm}
The rate function defined by~\eqref{eq:Cb-type_FL} is not necessarily convex unlike the Fenchel--Legendre transform.
The following theorem in the large deviation theory identifies the good convex rate function for an LDP as the Fenchel--Legendre transform of the limiting logarithmic moment generating function~(see also Theorem~\ref{thm:DZ_4.5.10_general} in Appendix~\ref{sec:LDP_rdm_meas} for a more general statement in the setting of a topological vector space instead of $\R^h$).
\begin{thm}[{\cite[Theorem~4.5.10]{DZ10}}]\label{thm:DZ_4.5.10}
Let $\{S^n\}_{n\in\N}$ be an $\R^h$-valued process.
Suppose that the $\R^h$-valued process $\{S^n/|\Lm^n|\}_{n\in\N}$ satisfies an LDP with speed $|\Lm^n|$ and a good convex rate function $I\colon\R^h\to[0,\infty]$, and also that
\[
\sup_{n\in\N}|\Lm^n|^{-1}\log\E[\exp(\langle\lm,S^n\rangle_{\R^h})]<\infty
\]
for any $\lm\in\R^h$.
Then, for every $\lm\in\R^h$, the limit
\[
\varphi(\lm)\coloneqq\lim_{n\to\infty}|\Lm^n|^{-1}\log\E[\exp(\langle\lm,S^n\rangle_{\R^h})]
\]
exists in $\R$, and $I=\varphi^*$ holds.
\end{thm}

In order to prove Theorem~\ref{thm:LDP_ernap}, what remains to be shown are the following: to check that the assumption of Theorem~\ref{thm:DZ_4.4.10} is satisfied, to verify the convexity of the good rate function obtained via Theorem~\ref{thm:DZ_4.4.10}, and to characterize the zero point of the good rate function.
Those are accomplished by the following three lemmas in this order.
These proofs are deferred to Subsection~\ref{ssec:pf_3lemmas}.
\begin{lem}\label{lem:lim_lEe_f}
Suppose that $\{S^n\}_{n\in\N}$ is an exponentially regular nearly additive $\R^h$-valued process satisfying that
\[
\sup_{n\in\N}|\Lm^n|^{-1}\log\E[\exp(\langle\lm,S^n\rangle_{\R^h})]<\infty
\]
for any $\lm\in\R^h$.
Then, $\{S^n/|\Lm^n|\}_{n\in\N}$ is exponentially tight, and the limit $\lim_{n\to\infty}|\Lm^n|^{-1}\Gm_n(f)$ exists in $[-\infty,\infty)$ for any $f\in\sF(\R^h)$.
\end{lem}
For the next lemma, we introduce a rate function associated with the empirical means of i.i.d. random variables.
For each $k\in\N$, let $\psi_k\colon\R^h\to(-\infty,\infty]$ be the logarithmic moment generating function of $S^k/|\Lm^k|$, i.e.,
\[
\psi_k(\lm)\coloneqq\log\E[\exp(\langle\lm,S^k/|\Lm^k|\rangle_{\R^h})]
\]
for any $\lm\in\R^h$, and let $J_k$ denote the Fenchel--Legendre transform of $\psi_k$.
Then, it is well known that, if we set $\{W_i\}_{i=1}^\infty$ as independent copies of $S^k/|\Lm^k|$, then
\[
\limsup_{m\to\infty}m^{-1}\log\P\biggl(\frac1m\sum_{i=1}^mW_i\in F\biggr)\le-\inf_{x\in F}J_k(x)
\]
for any closed set $F\subset\R^h$, and
\[
\liminf_{m\to\infty}m^{-1}\log\P\biggl(\frac1m\sum_{i=1}^mW_i\in G\biggr)\ge-\inf_{x\in G}J_k(x)
\]
for any open set $G\subset\R^h$ (see, e.g.,~\cite[Theorem~2.2.30]{DZ10}).
\begin{lem}\label{lem:convexity}
Let $\{S^n\}_{n\in\N}$ be an exponentially nearly additive $\R^h$-valued process.
Suppose that the $\R^h$-valued process $\{S^n/|\Lm^n|\}_{n\in\N}$ satisfies an LDP with speed $|\Lm^n|$ and a rate function $I\colon\R^h\to[0,\infty]$.
Then, the following hold.
\begin{enumerate}
\item For every $x\in\R^h$,
\begin{equation}\label{eq:convexity}
I(x)\le\liminf_{k\to\infty}\frac{J_k(x)}{|\Lm^k|}.
\end{equation}
\item $I$ is convex.
\end{enumerate}
\end{lem}

\begin{lem}\label{lem:zero_point}
Let $\{S^n\}_{n\in\N}$ be an exponentially nearly additive $\R^h$-valued process.
Suppose that the $\R^h$-valued process $\{S^n/|\Lm^n|\}_{n\in\N}$ satisfies an LDP with speed $|\Lm^n|$ and a rate function $I\colon\R^h\to[0,\infty]$.
Suppose also that the limit
\begin{equation}\label{eq:spatial_const}
\widehat S\coloneqq\lim_{n\to\infty}\frac{\E[S^n]}{|\Lm^n|}
\end{equation}
exists in $\R^h$, and that $S^n/|\Lm^n|$ converges to $\widehat S$ in probability as $n\to\infty$.
Then, $I(x)=0$ if and only if $x=\widehat S$.
\end{lem}

Combining these lemmas with the above two theorems and the strong law of large numbers for $\{S^n\}_{n\in\N}$, we can immediately prove Theorem~\ref{thm:LDP_ernap}.
\begin{proof}[Proof of Theorem~\ref{thm:LDP_ernap}]
From the assumption, Lemma~\ref{lem:lim_lEe_f} implies that $\{S^n/|\Lm^n|\}_{n\in\N}$ is exponentially tight, and the limit $\lim_{n\to\infty}|\Lm^n|^{-1}\Gm_n(f)$ exists in $[-\infty,\infty)$ for any $f\in\sF(\R^h)$.
By Theorem~\ref{thm:DZ_4.4.10}, for every $f\in C_b(\R^h)$, the limit
\[
\Gm(f)\coloneqq\lim_{n\to\infty}|\Lm^n|^{-1}\Gm_n(f)
\]
also exists in $\R$, and the $\R^h$-valued process $\{S^n/|\Lm^n|\}_{n\in\N}$ satisfies the LDP with speed $|\Lm^n|$ and a good rate function $I\colon\R^h\to[0,\infty]$ defined by
\[
I(x)\coloneqq\sup_{f\in C_b(\R^h)}\{f(x)-\Gm(f)\}.
\]
Furthermore, $I$ is convex by Lemma~\ref{lem:convexity}.
Therefore, it follows from Theorem~\ref{thm:DZ_4.5.10} that for every $\lm\in\R^h$, the limit
\[
\varphi(\lm)\coloneqq\lim_{n\to\infty}|\Lm^n|^{-1}\log\E[\exp(\langle\lm,S^n\rangle_{\R^h})]
\]
exists in $\R$, and $I=\varphi^*$ holds.
Lastly, combining Lemma~\ref{lem:zero_point} with the strong law of large numbers for $\{S^n\}_{n\in\N}$ discussed before Theorem~\ref{thm:LDP_ernap}, we conclude that $\varphi^*(x)=0$ if and only if $x=\widehat S$, which completes the proof.
\end{proof}

\subsection{Proofs of Lemmas~\ref{lem:lim_lEe_f},~\ref{lem:convexity}, and~\ref{lem:zero_point}}\label{ssec:pf_3lemmas}
In this subsection, we prove Lemmas~\ref{lem:lim_lEe_f},~\ref{lem:convexity}, and~\ref{lem:zero_point} in this order.
\begin{proof}[Proof of Lemma~\ref{lem:lim_lEe_f}]
We set $A(\lm)=\sup_{n\in\N}|\Lm^n|^{-1}\log\E[\exp(\langle\lm,S^n\rangle_{\R^h})]<\infty$ for any $\lm\in\R^h$, and write $S^n=(S^n_1,S^n_2,\ldots,S^n_h)$.
For any $\a\ge0$, the Markov inequality after exponentiating yields
\[
|\Lm^n|^{-1}\log\P\biggl(\sum_{i=1}^h\bigl|S^n_i\bigr|\ge \a|\Lm^n|\biggr)
\le|\Lm^n|^{-1}\log\E\biggl[\exp\biggl(\sum_{i=1}^h\bigl|S^n_i\bigr|\biggr)\biggr]-\a.
\]
Furthermore,
\begin{align}\label{eq:lim_lEe_f_1}
\E\biggl[\exp\biggl(\sum_{i=1}^h\bigl|S^n_i\bigr|\biggr)\biggr]
&=\sum_{J\subset[h]}\E\biggl[\exp\biggl(\sum_{i=1}^h\bigl|S^n_i\bigr|\biggr);S^n_j\ge0\text{ for $j\in J$},S^n_j<0\text{ for $j\notin J$}\biggr]\nonumber\\
&\le\sum_{J\subset[h]}\E[\exp(\langle\lm_J,S^n\rangle)]\nonumber\\
&\le\sum_{J\subset[h]}\exp(A(\lm_J)|\Lm^n|),
\end{align}
where $[h]\coloneqq\{1,2,\ldots,h\}$ and $\lm_J\in\R^h$ is a vector whose $j$th element is $1$ if $j\in J$, otherwise $-1$.
Combining the above estimates, we obtain
\[
\limsup_{n\to\infty}|\Lm^n|^{-1}\log\P\biggl(\sum_{i=1}^h\bigl|S^n_i\bigr|\ge\a|\Lm^n|\biggr)
\le\max_{J\subset[h]}A(\lm_J)-\a,
\]
which immediately implies the exponential tightness of $\{S^n/|\Lm^n|\}_{n\in\N}$.

Let $f\in\sF(\R^h)$ be fixed, and set
\[
\bar\Gm(f)
\coloneqq\limsup_{n\to\infty}|\Lm^n|^{-1}\Gm_n(f)
=\limsup_{n\to\infty}|\Lm^n|^{-1}\log\E\biggl[\exp\biggl(|\Lm^n|f\biggl(\frac{S^n}{|\Lm^n|}\biggr)\biggr)\biggr]\in[-\infty,\infty].
\]
If $\bar\Gm(f)=-\infty$, then there is nothing to prove.
Hence, we assume $\bar\Gm(f)>-\infty$.
Writing the Lipschitz constant of $f$ by $\|f\|_{\Lip}$, we have
\begin{align*}
\exp\biggl(\b|\Lm^n|f\biggl(\frac{S^n}{|\Lm^n|}\biggr)\biggr)
&\le\exp(\b|\Lm^n|f(0)+\b\|f\|_{\Lip}\bigl\|S^n\bigr\|_{\R^h})\\
&\le\exp(\b|\Lm^n|f(0))\exp\Biggl(\b\|f\|_{\Lip}\sum_{i=1}^h\bigl|S^n_i\bigr|\Biggr)
\end{align*}
for any $\b\ge0$.
Therefore, it follows from a similar calculation to~\eqref{eq:lim_lEe_f_1} that for any $\b\ge0$, there exists a constant $A_\b\ge0$ such that
\begin{equation}\label{eq:lim_lEe_f_2}
\E\biggl[\exp\biggl(\b|\Lm^n|f\biggl(\frac{S^n}{|\Lm^n|}\biggr)\biggr)\biggr]\le\exp(A_\b|\Lm^n|)
\end{equation}
for all $n\in\N$.
In particular, $\bar\Gm(f)\le A_1<\infty$ by taking $\b=1$.

Now, we let $\eps>0$ and show that
\begin{equation}\label{eq:lim_lEe_f_3}
\liminf_{n\to\infty}|\Lm^n|^{-1}\Gm_n(f)\ge\bar\Gm(f)-\eps.
\end{equation}
Set $\eps_0\coloneqq\eps/(2\|f\|_{\Lip}+1)$ and take a sufficiently large $C_0\ge0$ satisfying that $(A_2-C_0)/2<\bar\Gm(f)-\eps$.
By the exponentially near additivity of $\{S^n\}_{n\in\N}$, we can take an integer $r\ge0$, random variables $\{S^{n,z}\}_{n\in\N,z\in\Z^d}$, and an integer $K>r$ such that $\{S^{n,z}\}_{z\in\Z^d}$ are independent copies of $S^n$ for every $n\in\N$ and
\begin{equation}\label{eq:lim_lEe_f_4}
\P\biggl(\biggl\|S^{(2m+1)k}-\sum_{z\in\Z^d\cap[-m,m]^d}S^{k-r,z}\biggr\|_{\R^h}>\eps_0|\Lm^{(2m+1)k}|\biggr)\le\exp(-C_0|\Lm^{(2m+1)k}|)
\end{equation}
holds for all $k\ge K$ and $m\in\N$.
We may choose the integer $K$ large enough so that
\begin{equation}\label{eq:lim_lEe_f_5}
\frac{A_2-C_0}2<\biggl(1-\frac rK\biggr)^d\bar\Gm(f)-\eps.
\end{equation}
Now, we fix $k\ge K$ such that
\begin{equation}\label{eq:lim_lEe_f_6}
|\Lm^{k-r}|^{-1}\Gm_{k-r}(f)\ge\bar\Gm(f)-\eps_0.
\end{equation}
Let $m_n$ be the unique integer satisfying that $(2m_n+1)k\le n<(2m_n+3)k$ for each $n\in\N$.
Then, from the exponential regularity of $\{S^n\}_{n\in\N}$, we can take $N\in\N$ such that
\begin{equation}\label{eq:lim_lEe_f_7}
\P(\|S^n-S^{(2m_n+1)k}\|_{\R^h}>\eps_0|\Lm^n|)\le\exp(-C_0|\Lm^n|)
\end{equation}
for all $n\ge N$.
By the Lipschitzness of $f$, we have
\begin{align*}
f\biggl(\frac{S^n}{|\Lm^n|}\biggr)
&\ge f\biggl(\frac1{|\Lm^n|}\sum_{z\in\Z^d\cap[-m_n,m_n]^d}S^{k-r,z}\biggr)-\frac{\|f\|_{\Lip}}{|\Lm^n|}\biggl\|S^n-\sum_{z\in\Z^d\cap[-m_n,m_n]^d}S^{k-r,z}\biggr\|_{\R^h}\\
&\ge f\biggl(\frac1{|\Lm^n|}\sum_{z\in\Z^d\cap[-m_n,m_n]^d}S^{k-r,z}\biggr)-\frac{\|f\|_{\Lip}}{|\Lm^n|}(R_1(n)+R_2(n)),
\end{align*}
where
\[
R_1(n)\coloneqq\|S^n-S^{(2m_n+1)k}\|_{\R^h}\quad\text{and}\quad R_2(n)\coloneqq\biggl\|S^{(2m_n+1)k}-\sum_{z\in\Z^d\cap[-m_n,m_n]^d}S^{k-r,z}\biggr\|_{\R^h}.
\]
Furthermore, the concavity of $f$ yields
\begin{align*}
&f\biggl(\frac1{|\Lm^n|}\sum_{z\in\Z^d\cap[-m_n,m_n]^d}S^{k-r,z}\biggr)\\
&=f\biggl(\frac{(2m_n+1)^d|\Lm^{k-r}|}{|\Lm^n|}\frac1{(2m_n+1)^d}\sum_{z\in\Z^d\cap[-m_n,m_n]^d}\frac{S^{k-r,z}}{|\Lm^{k-r}|}\biggr)\\
&\ge\frac{(2m_n+1)^d|\Lm^{k-r}|}{|\Lm^n|}f\biggl(\frac1{(2m_n+1)^d}\sum_{z\in\Z^d\cap[-m_n,m_n]^d}\frac{S^{k-r,z}}{|\Lm^{k-r}|}\biggr)+\biggl(1-\frac{(2m_n+1)^d|\Lm^{k-r}|}{|\Lm^n|}\biggr)f(0)\\
&\ge\frac{|\Lm^{k-r}|}{|\Lm^n|}\sum_{z\in\Z^d\cap[-m_n,m_n]^d}f\biggl(\frac{S^{k-r,z}}{|\Lm^{k-r}|}\biggr)+\biggl(1-\frac{(2m_n+1)^d|\Lm^{k-r}|}{|\Lm^n|}\biggr)f(0).
\end{align*}
For the first inequality, we note that $(2m_n+1)^d|\Lm^{k-r}|\le|\Lm^{(2m_n+1)k}|\le|\Lm^n|$.
Combining the above estimates, we have
\begin{align*}
&|\Lm^n|f\biggl(\frac{S^n}{|\Lm^n|}\biggr)-\{|\Lm^n|-(2m_n+1)^d|\Lm^{k-r}|\}f(0)\\
&\ge|\Lm^{k-r}|\sum_{z\in\Z^d\cap[-m_n,m_n]^d}f\biggl(\frac{S^{k-r,z}}{|\Lm^{k-r}|}\biggr)-\|f\|_{\Lip}(R_1(n)+R_2(n)).
\end{align*}
By exponentiating and taking expectation,
\begin{align}\label{eq:lim_lEe_f_8}
&\E\biggl[\exp\biggl(|\Lm^n|f\biggl(\frac{S^n}{|\Lm^n|}\biggr)\biggr)\biggr]\cdot\exp(-\{|\Lm^n|-(2m_n+1)^d|\Lm^{k-r}|\}f(0))\nonumber\\
&\ge\E\biggl[\exp\biggl(|\Lm^{k-r}|\sum_{z\in\Z^d\cap[-m_n,m_n]^d}f\biggl(\frac{S^{k-r,z}}{|\Lm^{k-r}|}\biggr)\biggr)\cdot\exp(-\|f\|_{\Lip}(R_1(n)+R_2(n)))\biggr]\nonumber\\
&\ge\E\biggl[\exp\biggl(|\Lm^{k-r}|\sum_{z\in\Z^d\cap[-m_n,m_n]^d}f\biggl(\frac{S^{k-r,z}}{|\Lm^{k-r}|}\biggr)\biggr)\biggr]\cdot\exp(-\|f\|_{\Lip}\eps_0(|\Lm^n|+|\Lm^{(2m_n+1)k}|))\nonumber\\
&\qad-\E\biggl[\exp\biggl(|\Lm^{k-r}|\sum_{z\in\Z^d\cap[-m_n,m_n]^d}f\biggl(\frac{S^{k-r,z}}{|\Lm^{k-r}|}\biggr)\biggr);R_1(n)>\eps_0|\Lm^n|\biggr]\nonumber\\
&\qad-\E\biggl[\exp\biggl(|\Lm^{k-r}|\sum_{z\in\Z^d\cap[-m_n,m_n]^d}f\biggl(\frac{S^{k-r,z}}{|\Lm^{k-r}|}\biggr)\biggr);R_2(n)>\eps_0|\Lm^{(2m_n+1)k}|\biggr]\nonumber\\
&=:I_1-I_2-I_3.
\end{align}
For the following calculations, we note that $\{S^{n,z}\}_{z\in\Z^d}$ are independent copies of $S^n$ for every $n\in\N$.
Note also that $(2m_n+1)^d|\Lm^{k-r}|\le|\Lm^{(2m_n+1)k}|\le|\Lm^n|$.
By~\eqref{eq:lim_lEe_f_6},
\begin{align*}
I_1
&=\E\biggl[\exp\biggl(|\Lm^{k-r}|f\biggl(\frac{S^{k-r}}{|\Lm^{k-r}|}\biggr)\biggr)\biggr]^{(2m_n+1)^d}\cdot\exp(-\|f\|_{\Lip}\eps_0(|\Lm^n|+|\Lm^{(2m_n+1)k}|))\\
&\ge\exp((2m_n+1)^d|\Lm^{k-r}|(\bar\Gm(f)-\eps_0))\cdot\exp(-\|f\|_{\Lip}\eps_0(|\Lm^n|+|\Lm^{(2m_n+1)k}|))\\
&\ge\exp((2m_n+1)^d|\Lm^{k-r}|\bar\Gm(f)-(2\|f\|_{\Lip}+1)\eps_0|\Lm^n|)\\
&=\exp\biggl(\biggl(\frac{(2m_n+1)^d|\Lm^{k-r}|}{|\Lm^n|}\bar\Gm(f)-\eps\biggr)|\Lm^n|\biggr).
\end{align*}
By the Cauchy--Schwarz inequality,~\eqref{eq:lim_lEe_f_7}, and~\eqref{eq:lim_lEe_f_2} with $\b=2$,
\begin{align*}
I_2
&\le\E\biggl[\exp\biggl(2|\Lm^{k-r}|\sum_{z\in\Z^d\cap[-m_n,m_n]^d}f\biggl(\frac{S^{k-r,z}}{|\Lm^{k-r}|}\biggr)\biggr)\biggr]^{1/2}\cdot\P(R_1(n)>\eps_0|\Lm^n|)^{1/2}\\
&=\E\biggl[\exp\biggl(2|\Lm^{k-r}|f\biggl(\frac{S^{k-r}}{|\Lm^{k-r}|}\biggr)\biggr)\biggr]^{(2m_n+1)^d/2}\cdot\exp(-C_0|\Lm^n|/2)\\
&\le\exp\biggl(\frac{A_2-C_0}2|\Lm^n|\biggr)
\end{align*}
for all $n\ge N$.
Similarly, by the Cauchy--Schwarz inequality,~\eqref{eq:lim_lEe_f_4}, and~\eqref{eq:lim_lEe_f_2} with $\b=2$,
\begin{align*}
I_3
\le\exp(A_2|\Lm^n|/2-C_0|\Lm^{(2m_n+1)k}|/2)
=\exp\biggl(\frac12\biggl(A_2-C_0\frac{|\Lm^{(2m_n+1)k}|}{|\Lm^n|}\biggr)|\Lm^n|\biggr)
\end{align*}
for all $n\ge N$.
Since
\[
\lim_{n\to\infty}\frac{(2m_n+1)^d|\Lm^{k-r}|}{|\Lm^n|}
=\biggl(1-\frac rk\biggr)^d\quad\text{and}\quad\lim_{n\to\infty}\frac{|\Lm^{(2m_n+1)k}|}{|\Lm^n|}=1,
\]
it follows from~\eqref{eq:lim_lEe_f_5} that
\[
\frac{A_2-C_0}2<\frac12\biggl(A_2-C_0\frac{|\Lm^{(2m_n+1)k}|}{|\Lm^n|}\biggr)<\frac{(2m_n+1)^d|\Lm^{k-r}|}{|\Lm^n|}\bar\Gm(f)-\eps
\]
for sufficiently large $n$.
Therefore, an elementary calculation yields
\[
\liminf_{n\to\infty}|\Lm^n|^{-1}\log(I_1-I_2-I_3)
\ge\biggl(1-\frac rk\biggr)^d\bar\Gm(f)-\eps.
\]
Thus, combining this estimate with~\eqref{eq:lim_lEe_f_8}, we obtain
\[
\liminf_{n\to\infty}|\Lm^n|^{-1}\Gm_n(f)-\biggl(1-\biggl(1-\frac rk\biggr)^d\biggr)f(0)
\ge\biggl(1-\frac rk\biggr)^d\bar\Gm(f)-\eps.
\]
Taking $k\to\infty$ in the above inequality along a suitable subsequence so that~\eqref{eq:lim_lEe_f_6} is satisfied yields~\eqref{eq:lim_lEe_f_3}, which completes the proof since $\eps>0$ is arbitrary.
\end{proof}

Next, we turn to prove Lemma~\ref{lem:convexity}.
\begin{proof}[Proof of Lemma~\ref{lem:convexity}]
(1)~If $\liminf_{k\to\infty}J_k(x)/|\Lm^k|=\infty$, then there is nothing to prove.
Hence, we assume that $\liminf_{k\to\infty}J_k(x)/|\Lm^k|<\infty$.
Let $c<I(x)$ be fixed.
By the lower semicontinuity of the rate function $I$, we take $\eps>0$ such that $y\in\bar B(x,\eps)$ implies $I(y)\ge c$.
We set $C\coloneqq\liminf_{k\to\infty}J_k(x)/|\Lm^k|+1<\infty$.
From the exponentially near additivity of $\{S^n\}_{n\in\N}$, we can take an integer $r\ge0$, random variables $\{S^{n,z}\}_{n\in\N,z\in\Z^d}$, and an integer $K>r$ such that $\{S^{n,z}\}_{z\in\Z^d}$ are independent copies of $S^n$ for every $n\in\N$ and
\begin{equation}\label{eq:convexity_1_1}
\P\biggl(\biggl\|S^{(2m+1)k}-\sum_{z\in\Z^d\cap[-m,m]^d}S^{k-r,z}\biggr\|_{\R^h}>\frac\eps2|\Lm^{(2m+1)k}|\biggr)\le\exp(-C|\Lm^{(2m+1)k}|)
\end{equation}
holds for all $k\ge K$ and $m\in\N$.
We may choose the integer $K$ large enough so that
\begin{equation}\label{eq:convexity_1_2}
\frac{|\Lm^{K-r}|}{|\Lm^K|}x\in B(x,\eps/2).
\end{equation}
Now, we fix $k\ge K$ satisfying that
\begin{equation}\label{eq:convexity_1_3}
\frac{J_{k-r}(x)+1}{|\Lm^{k-r}|}<C.
\end{equation}
By the large deviation principles for $\{S^n/|\Lm^n|\}_{n\in\N}$ controlled by $I$, we have
\begin{align}\label{eq:convexity_1_4}
-c
&\ge-\inf_{y\in\bar B(x,\eps)}I(y)\nonumber\nonumber\\
&\ge\limsup_{m\to\infty}\frac1{|\Lm^{(2m+1)k}|}\log\P\biggl(\frac{S^{(2m+1)k}}{|\Lm^{(2m+1)k}|}\in\bar B(x,\eps)\biggr)\nonumber\\
&\ge\limsup_{m\to\infty}\frac1{|\Lm^{(2m+1)k}|}\log\biggl\{\P\biggl(\frac1{|\Lm^{(2m+1)k}|}\sum_{z\in\Z^d\cap[-m,m]^d}S^{k-r,z}\in B(x,\eps/2)\biggr)\nonumber\\
&\qad-\P\biggl(\biggl\|S^{(2m+1)k}-\sum_{z\in\Z^d\cap[-m,m]^d}S^{k-r,z}\biggr\|_{\R^h}>\frac\eps2|\Lm^{(2m+1)k}|\biggr)\biggr\}\nonumber\\
&\ge\limsup_{m\to\infty}\frac1{|\Lm^{(2m+1)k}|}\log\biggl\{\P\biggl(\frac1{(2m+1)^d}\sum_{z\in\Z^d\cap[-m,m]^d}\frac{S^{k-r,z}}{|\Lm^{k-r}|}\in\frac{|\Lm^k|}{|\Lm^{k-r}|}B(x,\eps/2)\biggr)\nonumber\\
&\qad-\exp(-C|\Lm^{(2m+1)k}|)\biggr\}.
\end{align}
The fourth inequality was obtained by~\eqref{eq:convexity_1_1}.
In order to estimate the right-hand side of~\eqref{eq:convexity_1_4}, we will use the lower bound of Cram\'er's large deviation theorem for
\[
\Biggl\{\frac1{(2m+1)^d}\sum_{z\in\Z^d\cap[-m,m]^d}\frac{S^{k-r,z}}{|\Lm^{k-r}|}\Biggr\}_{m\in\N}
\]
controlled by $J_{k-r}$:
\begin{align*}
&\liminf_{m\to\infty}\frac1{(2m+1)^d}\log\P\biggl(\frac1{(2m+1)^d}\sum_{z\in\Z^d\cap[-m,m]^d}\frac{S^{k-r,z}}{|\Lm^{k-r}|}\in\frac{|\Lm^k|}{|\Lm^{k-r}|}B(x,\eps/2)\biggr)\\
&\ge-\inf\left\{J_{k-r}(y)\relmiddle|y\in\frac{|\Lm^k|}{|\Lm^{k-r}|}B(x,\eps/2)\right\}.
\end{align*}
Since the right-hand side of the above inequality is bounded below by $-J_{k-r}(x)$ from~\eqref{eq:convexity_1_2}, we obtain
\begin{equation}\label{eq:convexity_1_5}
\P\biggl(\frac1{(2m+1)^d}\sum_{z\in\Z^d\cap[-m,m]^d}\frac{S^{k-r,z}}{|\Lm^{k-r}|}\in\frac{|\Lm^k|}{|\Lm^{k-r}|}B(x,\eps/2)\biggr)
\ge\exp(-\{J_{k-r}(x)+1\}(2m+1)^d)
\end{equation}
for sufficiently large $m$.
Combining~\eqref{eq:convexity_1_4} and~\eqref{eq:convexity_1_5},
\begin{align*}
-c
&\ge\limsup_{m\to\infty}\frac1{|\Lm^{(2m+1)k}|}\log\biggl\{\exp(-\{J_{k-r}(x)+1\}(2m+1)^d)-\exp(-C|\Lm^{(2m+1)k}|)\biggr\}\nonumber\\
&=\limsup_{m\to\infty}\frac1{|\Lm^{(2m+1)k}|}\log\biggl\{\exp\biggl(-\frac{J_{k-r}(x)+1}{|\Lm^k|}|\Lm^{(2m+1)k}|\biggr)-\exp(-C|\Lm^{(2m+1)k}|)\biggr\}\nonumber\\
&=-\frac{J_{k-r}(x)+1}{|\Lm^k|}.
\end{align*}
For the last line, we note~\eqref{eq:convexity_1_3}.
Letting $k\to\infty$ in the above inequality along a suitable subsequence so that~\eqref{eq:convexity_1_3} is satisfied, we conclude that $c\le\liminf_{k\to\infty}J_k(x)/|\Lm^k|$, which completes the proof since $c<I(x)$ is arbitrary.

(2)~Let $x_1\neq x_2\in\R^h$ and $t\in(0,1)$ be fixed, and set $x\coloneqq tx_1+(1-t)x_2$.
We will show
\begin{equation}\label{eq:convexity_2_1}
tI(x_1)+(1-t)I(x_2)\ge I(x).
\end{equation}
We may assume that $I(x_1)<\infty$ and $I(x_2)<\infty$.
We first take an integer $r\ge0$ and random variables $\{S^{n,z}\}_{n\in\N,z\in\Z^d}$ in the definition of the exponentially near additivity of $\{S^n\}_{n\in\N}$.
Let $l\in\N$ be fixed.
We set $\dl>0$ such that $ty_1+(1-t)y_2\in\bar B(x,l^{-1})$ for any $y_1\in\bar B(x_1,\dl)$ and $y_2\in\bar B(x_2,\dl)$.
From the exponentially near additivity of $\{S^n\}_{n\in\N}$, letting $C>\max\{I(x_1),I(x_2)\}$, we can take an integer $K_l>r$ such that
\begin{equation}\label{eq:convexity_2_2}
\P\biggl(\biggl\|S^{(2m+1)k}-\sum_{z\in\Z^d\cap[-m,m]^d}S^{k-r,z}\biggr\|_{\R^h}>\frac\dl3|\Lm^{(2m+1)k}|\biggr)\le\exp(-C|\Lm^{(2m+1)k}|)
\end{equation}
holds for all $k\ge K_l$ and $m\in\N$.
We may choose the integer $K_l$ large enough so that for $i=1,2$,
\begin{equation}\label{eq:convexity_2_3}
\frac{|\Lm^{K_l}|}{|\Lm^{K_l-r}|}B(x_i,2\dl/3)\subset\bar B(x_i,\dl).
\end{equation}
Then, by the large deviation principle for $\{S^n/|\Lm^n|\}_{n\in\N}$ controlled by $I$ and Cram\'er's large deviation theorem for
\[
\Biggl\{\frac1{(2m+1)^d}\sum_{z\in\Z^d\cap[-m,m]^d}\frac{S^{k-r,z}}{|\Lm^{k-r}|}\Biggr\}_{m\in\N}
\]
controlled by $J_{k-r}$, for any $k\ge K_l$ and $i=1,2$,
\begin{align*}
&-I(x_i)\\
&\le-\inf_{y\in B(x_i,\dl/3)}I(y)\\
&\le\liminf_{m\to\infty}\frac1{|\Lm^{(2m+1)k}|}\log\P\biggl(\frac{S^{(2m+1)k}}{|\Lm^{(2m+1)k}|}\in B(x_i,\dl/3)\biggr)\\
&\le\limsup_{m\to\infty}\frac1{|\Lm^{(2m+1)k}|}\log\biggl[\P\biggl(\frac1{(2m+1)^d}\sum_{z\in\Z^d\cap[-m,m]^d}\frac{S^{k-r,z}}{|\Lm^{k-r}|}\in\frac{|\Lm^k|}{|\Lm^{k-r}|}B(x_i,2\dl/3)\biggr)\\
&\qad+\P\biggl(\biggl\|S^{(2m+1)k}-\sum_{z\in\Z^d\cap[-m,m]^d}S^{k-r,z}\biggr\|_{\R^h}>\frac\dl3|\Lm^{(2m+1)k}|\biggr)\biggr]\\
&\le\limsup_{m\to\infty}\frac1{|\Lm^{(2m+1)k}|}\log\biggl[\P\biggl(\frac1{(2m+1)^d}\sum_{z\in\Z^d\cap[-m,m]^d}\frac{S^{k-r,z}}{|\Lm^{k-r}|}\in\bar B(x_i,\dl)\biggr)+\exp(-C|\Lm^{(2m+1)k}|)\biggr]\\
&\le\limsup_{m\to\infty}\frac1{|\Lm^{(2m+1)k}|}\log\Bigl[\exp\Bigl(-\Bigl\{\inf_{y\in\bar B(x_i,\dl)}J_{k-r}(y)-1\Bigr\}(2m+1)^d\Bigr)+\exp(-C|\Lm^{(2m+1)k}|)\Bigr]\\
&=-\min\biggl\{\frac{\inf_{y\in\bar B(x_i,\dl)}J_{k-r}(y)-1}{|\Lm^k|},C\biggr\}.
\end{align*}
The fourth inequality follows from~\eqref{eq:convexity_2_2} and~\eqref{eq:convexity_2_3}.
Since $C>\max\{I(x_1),I(x_2)\}$, we obtain
\begin{equation}\label{eq:convexity_2_4}
I(x_i)\ge\frac{\inf_{y\in\bar B(x_i,\dl)}J_{k-r}(y)-1}{|\Lm^k|}
\end{equation}
for any $k\ge K_l$ and $i=1,2$.
Therefore, it follows from the convexity of $J_{k-r}$ and the setting of $\dl$ that
\begin{align*}
tI(x_1)+(1-t)I(x_2)
&\ge\frac{t\inf_{y\in\bar B(x_1,\dl)}J_{k-r}(y)+(1-t)\inf_{y\in\bar B(x_2,\dl)}J_{k-r}(y)-1}{|\Lm^k|}\\
&\ge\frac{\inf_{y\in\bar B(x,\eps)}J_{k-r}(y)-1}{|\Lm^k|}
\end{align*}
for any $k\ge K_l$.
From the lower semicontinuity of $J_{k-r}$, there exists $y_l\in\bar B(x,l^{-1})$ such that
\[
J_{k-r}(y_l)=\inf_{y\in\bar B(x,\eps)}J_{k-r}(y).
\]
From the above discussion, there exists a sequence $\{K_l\}_{l\in\N}$ of integers $K_l>r$ and $\{y_l\}_{l\in\N}\subset\R^h$ such that $\lim_{l\to\infty}y_l=x$ and
\[
tI(x_1)+(1-t)I(x_2)
\ge\frac{J_{k-r}(y_l)-1}{|\Lm^k|}
\]
for any $l\in\N$ and $k\ge K_l$.
By taking $k\to\infty$ in the above inequality and using~\eqref{eq:convexity}, we have
\[
tI(x_1)+(1-t)I(x_2)
\ge\liminf_{k\to\infty}\frac{J_k(y_l)}{|\Lm^k|}
\ge I(y_l)
\]
for any $l\in\N$.
Letting $l\to\infty$ in the above inequality, we obtain~\eqref{eq:convexity_2_1} from the lower semicontinuity of $I$, which completes the proof.
\end{proof}

Lastly, we prove Lemma~\ref{lem:zero_point}.
\begin{proof}[Proof of Lemma~\ref{lem:zero_point}]
We fist show $I(\widehat S)=0$.
Let $\eps>0$ be fixed.
From the lower semicontinuity of $I$, there exists $\dl>0$ such that $x\in\bar B(\widehat S,\dl)$ implies $I(x)\ge I(\widehat S)-\eps$.
Then, by the assumption,
\[
0=\limsup_{n\to\infty}\log\P\biggl(\frac{S^n}{|\Lm^n|}\in\bar B(\widehat S,\dl)\biggr)
\le-\inf_{x\in\bar B(\widehat S,\dl)}I(x)
\le-I(\widehat S)+\eps.
\]
Since $\eps>0$ is arbitrary, we obtain $I(\widehat S)=0$.

Next, suppose that $I(x)=0$.
Let $\dl>0$ be fixed.
By almost the same calculation to obtain~\eqref{eq:convexity_2_4}, we can conclude that there exists $K\in\N$ such that
\[
0=I(x)\ge\frac{\inf_{y\in\bar B(x,\dl)}J_{k-r}(y)-\rho}{|\Lm^k|}
\]
for any $k\ge K$ and $\rho>0$.
Since $\rho>0$ is arbitrary, for any $k\ge K$,
\[
\inf_{y\in\bar B(x,\dl)}J_{k-r}(y)=0,
\]
which implies that $\E[S^{k-r}]/|\Lm^{k-r}|\in\bar B(x,\dl)$ since $\E[S^{k-r}]/|\Lm^{k-r}|$ is a unique zero point of the rate function $J_{k-r}$.
Using~\eqref{eq:spatial_const}, we obtain $\widehat S\in\bar B(x,\dl)$.
Since $\dl>0$ is arbitrary, it must be $x=\widehat S$.
\end{proof}

\section{Large deviation principle for persistent Betti numbers}\label{sec:LDP_PB}
In this section, we prove Theorems~\ref{thm:LLN_PB} and~\ref{thm:LDP_PBtuple}.
In Subsection~\ref{ssec:property_PB}, we estimate the difference of the persistent Betti numbers of two cubical filtrations in $\R^d$.
In Subsection~\ref{ssec:LLN-LDP_PB}, we check the exponential regularity and exponentially near additivity of persistent Betti numbers using the estimate, and prove Theorems~\ref{thm:LLN_PB} and~\ref{thm:LDP_PBtuple} by applying Corollary~\ref{cor:LDP_ernap}.

\subsection{Properties of persistent Betti number}\label{ssec:property_PB}
We start with the simple bound of the persistent Betti number of a cubical filtration in $\R^d$.
\begin{prop}\label{prop:triv_bound_PB}
Let $\X=\{X(t)\}_{t\ge0}$ be a bounded cubical filtration in $\R^d$.
Fix an integer $0\le q<d$ and $0\le s\le t<\infty$.
Then,
\begin{equation}\label{eq:triv_bound_PB}
\b_q^{\X}(s,t)\le\b_q(X(s))\le\#\cK^d_q(X(s)).
\end{equation}
\end{prop}
\begin{proof}
Since $Z_q(X(s))\cap B_q(X(t))\supset Z_q(X(s))\cap B_q(X(s))=B_q(X(s))$, it follows from~\eqref{eq:PB} that
\[
\b_q^\X(s,t)
=\dim\frac{Z_q(X(s))}{Z_q(X(s))\cap B_q(X(t))}
\le\dim\frac{Z_q(X(s))}{B_q(X(s))}
=\b_q(X(s)).
\]
Furthermore,
\[
\b_q(X(s))
=\dim\frac{Z_q(X(s))}{B_q(X(s))}
\le\dim Z_q(X(s))
\le\dim C_q(X(s))
=\#\cK^d_q(X(s)).\qedhere
\]
\end{proof}

Next, we estimate the difference of the persistent Betti numbers of two cubical filtrations in $\R^d$.
We will use the following basic fact in linear algebra.
\begin{lem}\label{lem:basic_linear}
Let
\[
D=\left[\begin{array}{cc}
A &B\\
0 &C
\end{array}\right]
\]
be a block matrix, and let $l$ be the number of columns in $B$ and $C$.
Then,
\[
\rank A\le\rank D\le\rank A+l
\quad\text{and}\quad
\dim\ker A\le\dim\ker D\le\dim\ker A+l.
\]
\end{lem}
\begin{proof}
Since the rank coincides with the number of linearly independent columns, the first conclusion is trivial.
The second conclusion follows immediately from the first conclusion with the rank-nullity theorem: letting $k$ be the number of columns in $A$,
\[
\rank D+\dim\ker D=k+l
\quad\text{and}\quad
\rank A+\dim\ker A=k.\qedhere
\]
\end{proof}
The following is a generalization of Lemma~3.1 of~\cite{HT18} to persistent Betti numbers as well as an analogous result of Lemma~2.11 in~\cite{HST18}, where they consider filtrations of simplicial complexes.
\begin{prop}\label{prop:PB_diff}
Let $\X=\{X(t)\}_{t\ge0}$ and $\Y=\{Y(t)\}_{t\ge0}$ be bounded cubical filtrations in $\R^d$ with $X(t)\subset Y(t)$ for every $t\ge0$.
Fix an integer $0\le q<d$ and $0\le s\le t<\infty$.
Then,
\begin{equation}\label{eq:PB_diff}
|\b_q^{\Y}(s,t)-\b_q^\X(s,t)|
\le\#(\cK^d_q(Y(s))\setminus\cK^d_q(X(s)))+\#(\cK^d_{q+1}(Y(t))\setminus\cK^d_{q+1}(X(t))).
\end{equation}
\end{prop}
\begin{proof}
From~\eqref{eq:PB},
\begin{align*}
\b_q^{\Y}(s,t)-\b_q^\X(s,t)
&=\{\dim(Z_q(Y(s)))-\dim(Z_q(Y(s))\cap B_q(Y(t)))\}\\
&\qad-\{\dim(Z_q(X(s)))-\dim(Z_q(X(s))\cap B_q(X(t)))\}\\
&=\dim\biggl(\frac{Z_q(Y(s))}{Z_q(X(s))}\biggr)-\dim\biggl(\frac{Z_q(Y(s))\cap B_q(Y(t))}{Z_q(X(s))\cap B_q(X(t))}\biggr).
\end{align*}
Since
\[
\dim\biggl(\frac{Z_q(Y(s))\cap B_q(Y(t))}{Z_q(X(s))\cap B_q(X(t))}\biggr)
\le\dim\biggl(\frac{Z_q(Y(s))}{Z_q(X(s))}\biggr)+\dim\biggl(\frac{B_q(Y(t))}{B_q(X(t))}\biggr)
\]
from an elementary calculation, we obtain
\[
|\b_q^{\Y}(s,t)-\b_q^\X(s,t)|
\le\dim\biggl(\frac{Z_q(Y(s))}{Z_q(X(s))}\biggr)+\dim\biggl(\frac{B_q(Y(t))}{B_q(X(t))}\biggr).
\]
Using Lemma~\ref{lem:basic_linear}, we have
\begin{align*}
\dim\biggl(\frac{Z_q(Y(s))}{Z_q(X(s))}\biggr)
&=\dim\ker\del_q^{Y(s)}-\dim\ker\del_q^{X(s)}
\le\#(\cK^d_q(Y(s))\setminus\cK^d_q(X(s)))
\shortintertext{and}
\dim\biggl(\frac{B_q(Y(t))}{B_q(X(t))}\biggr)
&=\rank\del_{q+1}^{Y(t)}-\rank\del_{q+1}^{X(t)}
\le\#(\cK^d_{q+1}(Y(t))\setminus\cK^d_{q+1}(X(t))).
\end{align*}
Combining the above estimates, the conclusion follows.
\end{proof}
The following is useful when we estimate the right-hand sides of~\eqref{eq:triv_bound_PB} and~\eqref{eq:PB_diff}.
This is an easy consequence of the fact that each elementary $d$-cube contains exactly $\binom dq2^{d-q}$ number of elementary $q$-cubes for each integer $0\le q\le d$.
\begin{lem}[{\cite[Lemma~3.2]{HT18}}]\label{lem:no_cube-vol}
Let $X$ and $Y$ be bounded cubical sets with $X\subset Y$.
Fix an integer $0\le q\le d$.
Suppose that the subset $Y\setminus X\subset\R^d$ is covered by $v$ number of elementary $d$-cubes.
Then, $\#(\cK^d_q(Y)\setminus\cK^d_q(X))\le\binom dq2^{d-q}v$.
In particular, $\#(\cK^d_q(Y)\setminus\cK^d_q(X))\le\#(\cK^d(Y)\setminus\cK^d(X))\le3^dv$.
\end{lem}

\subsection{Proofs of Theorems~\ref{thm:LLN_PB} and~\ref{thm:LDP_PBtuple}}\label{ssec:LLN-LDP_PB}
In this subsection, let $\X=\{X(t)\}_{t\ge0}$ be a random cubical filtration in $\R^d$ satisfying Assumptions~\ref{asm:stationarity} and~\ref{asm:loc_dep}.
We fix an integer $0\le q<d$ and a finite family $\{(s_i,t_i)\}_{i=1}^h$ with $0\le s_i\le t_i<\infty$, and write $S^n=(\b_q^{\X^n}(s_1,t_1),\ldots,\b_q^{\X^n}(s_h,t_h))$ for each $n\in\N$.
We herein prove Theorems~\ref{thm:LLN_PB} and~\ref{thm:LDP_PBtuple} by applying Corollary~\ref{cor:LDP_ernap}. 
\begin{proof}[Proofs of Theorems~\ref{thm:LLN_PB} and~\ref{thm:LDP_PBtuple}]
By the Cauchy--Schwarz inequality, we have
\begin{align*}
|\Lm^n|^{-1}\log\E[\exp(\langle\lm,S^n\rangle_{\R^h})]
&\le|\Lm^n|^{-1}\log\E\biggl[\exp\biggl(\|\lm\|_{\R^h}\biggl(\sum_{i=1}^h\b_q^{\X^n}(s_i,t_i)^2\biggr)^{1/2}\biggr)\biggr]\\
&\le\sqrt h\|\lm\|_{\R^h}\frac{\#\cK^d_q(\Lm^n)}{|\Lm^n|}\\
&\le3^d\sqrt h\|\lm\|_{\R^h}
\end{align*}
for any $\lm=(\lm_1,\ldots,\lm_h)\in\R^h$.
Here, the second inequality follows from Proposition~\ref{prop:triv_bound_PB}.
For the third inequality, we used Lemma~\ref{lem:no_cube-vol}.
In order to apply Corollary~\ref{cor:LDP_ernap} to $\{S^n\}_{n\in\N}$, we additionally require the exponential regularity and exponentially near additivity of $\{S^n\}_{n\in\N}$.

We first show the exponential regularity of $\{S^n\}_{n\in\N}$.
Let $k\in\N$ be fixed, and let $m_n$ be the unique integer satisfying $(2m_n+1)k\le n<(2m_n+3)k$ for each $n\in\N$.
From Proposition~\ref{prop:PB_diff} and Lemma~\ref{lem:no_cube-vol}, it holds that
\begin{align*}
&\bigl\|S^n-S^{(2m_n+1)k}\bigr\|_{\R^h}\\
&=\biggl(\sum_{i=1}^h\bigl\{\b_q^{\X^n}(s_i,t_i)-\b_q^{\X^{(2m_n+1)k}}(s_i,t_i)\bigr\}^2\biggr)^{1/2}\\
&\le\biggl(\sum_{i=1}^h\bigl\{\#(\cK^d_q(X^n(s_i))\setminus\cK^d_q(X^{(2m_n+1)k}(s_i)))+\#(\cK^d_{q+1}(X^n(t_i))\setminus\cK^d_{q+1}(X^{(2m_n+1)k}(t_i)))\bigr\}^2\biggr)^{1/2}\\
&\le\sqrt h\biggl\{\binom dq2^{d-q}+\binom d{q+1}2^{d-(q+1)}\biggr\}(|\Lm^n|-|\Lm^{(2m_n+1)k}|)\\
&\le3^d\sqrt h(|\Lm^n|-|\Lm^{(2m_n+1)k}|).
\end{align*}
Therefore, we obtain
\[
|\Lm^n|^{-1}\bigl\|S^n-S^{(2m_n+1)k}\bigr\|_{\R^h}
\le3^d\sqrt h\{1-((2m_n+1)k/n)^d\}.
\]
The right-hand side of the above inequality converges to zero as $n\to\infty$, which immediately implies the exponential regularity of $\{S^n\}_{n\in\N}$.

We turn to prove the exponentially near additivity of $\{S^n\}_{n\in\N}$.
We take the integer $R\ge0$ as in Assumption~\ref{asm:loc_dep}, and choose an integer $r$ such that $2r>R$.
For each $n\in\N$ and $z\in\Z^d$, define
\[
\Lm^{n,z}\coloneqq2(n+r)z+\Lm^n\text{, }\X^{n,z}\coloneqq\{X(t)\cap\Lm^{n,z}\}_{t\ge0}\text{, and }S^{n,z}\coloneqq(\b_q^{\X^{n,z}}(s_1,t_1),\ldots,\b_q^{\X^{n,z}}(s_h,t_h)).
\]
Note that $d_{\max}(\Lm^{n,z},\Lm^{n,z'})\ge2r>R$ for any distinct $z,z'\in\Z^d$.
Therefore, $\{S^{n,z}\}_{z\in\Z^d}$ are independent copies of $S^n$ for every $n\in\N$ from Assumptions~\ref{asm:stationarity} and~\ref{asm:loc_dep}.
For each $k,m\in\N$, we also define
\[
\widetilde\Lm^{(2m+1)k}\coloneqq\bigsqcup_{z\in\Z^d\cap[-m,m]^d}\Lm^{k-r,z}
\quad\text{and}\quad
\widetilde\X^{(2m+1)k}\coloneqq\{X(t)\cap\widetilde\Lm^{(2m+1)k}\}_{t\ge0}.
\]
Since $\widetilde\Lm^{(2m+1)k}$ is a disjoint union of $\Lm^{k-r,z}$'s, we have
\[
\b_q^{\widetilde\X^{(2m+1)k}}(s_i,t_i)
=\sum_{z\in\Z^d\cap[-m,m]^d}\b_q^{\X^{k-r,z}}(s_i,t_i)
\]
for every $i=1,2,\ldots,h$.
Therefore, again from Proposition~\ref{prop:PB_diff} and Lemma~\ref{lem:no_cube-vol}, it follows that
\begin{align*}
&\biggl\|S^{(2m+1)k}-\sum_{z\in\Z^d\cap[-m,m]^d}S^{k-r,z}\biggr\|_{\R^h}\\
&=\Bigl\|(\b_q^{\X^{(2m+1)k}}(s_1,t_1),\ldots,\b_q^{\X^{(2m+1)k}}(s_h,t_h))-(\b_q^{\widetilde\X^{(2m+1)k}}(s_1,t_1),\ldots,\b_q^{\widetilde\X^{(2m+1)k}}(s_h,t_h))\Bigr\|_{\R^h}\\
&=\biggl(\sum_{i=1}^h\bigl\{\b_q^{\X^{(2m+1)k}}(s_i,t_i)-\b_q^{\widetilde\X^{(2m+1)k}}(s_i,t_i)\bigr\}^2\biggr)^{1/2}\\
&\le\sqrt h\biggl\{\binom dq2^{d-q}+\binom d{q+1}2^{d-(q+1)}\biggr\}(|\Lm^{(2m+1)k}|-|\widetilde\Lm^{(2m+1)k}|)\\
&\le3^d\sqrt h(|\Lm^{(2m+1)k}|-|\widetilde\Lm^{(2m+1)k}|).
\end{align*}
Therefore, we obtain
\[
|\Lm^{(2m+1)k}|^{-1}\biggl\|S^{(2m+1)k}-\sum_{z\in\Z^d\cap[-m,m]^d}S^{k-r,z}\biggr\|_{\R^h}
\le3^d\sqrt h\biggl(1-\frac{|\widetilde\Lm^{(2m+1)k}|}{|\Lm^{(2m+1)k}|}\biggr)
\le3^d\sqrt h\biggl(1-\biggl(1-\frac rk\biggr)^d\biggr).
\]
Since the right-hand side of the above inequality converges to zero as $k\to\infty$, the exponentially near additivity of $\{S^n\}_{n\in\N}$ follows.

Consequently, applying Corollary~\ref{cor:LDP_ernap} with $S^n=(\b_q^{\X^n}(s_1,t_1),\ldots,\b_q^{\X^n}(s_h,t_h))$, we complete the proofs of both Theorems~\ref{thm:LLN_PB} and~\ref{thm:LDP_PBtuple}.
\end{proof}

\section{Large deviation principle for persistence diagrams}\label{sec:LDP_PDs}
In this section, we prove Theorems~\ref{thm:LLN_PD} and~\ref{thm:LDP_PD}.
In Subsection~\ref{ssec:stmt_rslt_PD}, we briefly describe a method of lifting the strong law of large numbers for persistent Betti numbers to persistence diagrams, developed in~\cite{HST18}.
Moreover, we develop a general method of lifting an LDP for the tuples of persistent Betti numbers to persistence diagrams (Theorem~\ref{thm:LDP_PBtuple-PD}).
Applying those methods, we prove Theorems~\ref{thm:LLN_PD} and~\ref{thm:LDP_PD} in Subsection~\ref{ssec:pf_LLN-LDP_PD}.
The proof of Theorem~\ref{thm:LDP_PBtuple-PD} is deferred to Subsection~\ref{ssec:pf_PBtuple-PD}. 

\subsection{Statement of result}\label{ssec:stmt_rslt_PD}
Before proceeding to our LDP result, we introduce a method of lifting the strong law of large numbers for persistent Betti numbers to persistence diagrams, developed in~\cite[Section~3 and Appendix~A]{HST18}.
The following theorem is immediately obtained by combining~\cite[Proposition~3.4]{HST18} and~\cite[Corollary~A.3]{HST18} together with the inclusion-exclusion principle.
\begin{thm}\label{thm:LLN_PB-PD}
Let $\{\xi^n\}_{n\in\N}$ be an $\cM(\Dl)$-valued process.
Assume that $\E[\xi^n]\in\cM(\Dl)$ for all $n\in\N$ and that for any $0\le s<t<\infty$, the limit
\[
c_{s,t}\coloneqq\lim_{n\to\infty}\E[\xi^n([0,s]\times(t,\infty])]
\]
exists in $[0,\infty)$.
Then, there exists a Radon measure $\widehat\xi\in\cM(\Dl)$ such that $\E[\xi^n]$ converges vaguely to $\widehat\xi$ as $n\to\infty$.
Assume further that for any $0\le s<t<\infty$,
\[
\xi^n([0,s]\times(t,\infty])\to c_{s,t}\quad\text{almost surely as $n\to\infty$.}
\]
Then, $\xi^n$ converges vaguely to $\widehat\xi$ almost surely as $n\to\infty$.
\end{thm}

Next, we state our LDP result, which is useful to lift an LDP for the tuples of persistent Betti numbers to persistence diagrams.
In what follows in this section, let $\{a_n\}_{n\in\N}$ be a sequence of positive numbers tending to infinity as $n\to\infty$.
The following is the main result in this section.
\begin{thm}\label{thm:LDP_PBtuple-PD}
Let $\{\xi^n\}_{n\in\N}$ be an $\cM(\Dl)$-valued process.
Assume that
\begin{equation}\label{eq:LDP_PBtuple-PD}
\sup_{n\in\N}a_n^{-1}\log\E[\exp(a_n\lm\xi^n(\Dl))]<\infty
\end{equation}
for any $\lm>0$.
Assume further that for any finite family $\cP=\{(s_i,t_i)\}_{i=1}^h$ with $0\le s_i<t_i<\infty$, the $\R^h$-valued process
\[
\{(\xi^n([0,s_1]\times(t_1,\infty]),\xi^n([0,s_2]\times(t_2,\infty]),\ldots,\xi^n([0,s_h]\times(t_h,\infty]))\}_{n\in\N}
\]
satisfies an LDP with speed $a_n$ and a good rate function $I_\cP\colon\R^h\to[0,\infty]$.
Then, the $\cM(\Dl)$-valued process $\{\xi^n\}_{n\in\N}$ satisfies an LDP with speed $a_n$ and a good rate function $I\colon\cM(\Dl)\to[0,\infty]$.
Furthermore, the following statements hold.
\begin{enumerate}
\item Suppose that there exists a constant $K>0$ such that $\xi^n(\Dl)\le K$ for all $n\in\N$.
If $I_\cP$ has a unique zero point for each finite family $\cP=\{(s_i,t_i)\}_{i=1}^h$ with $0\le s_i<t_i<\infty$, then so does $I$.
\item If $I_\cP$ is convex for each finite family $\cP=\{(s_i,t_i)\}_{i=1}^h$ with $0\le s_i<t_i<\infty$, then for every $f\in C_c(\Dl)$, the limit
\[
\varphi(f)\coloneqq\lim_{n\to\infty}a_n^{-1}\log\E\biggl[\exp\biggl(a_n\int_\Dl f\,d\xi^n\biggr)\biggr]
\]
exists in $\R$, and it holds that
\[
I(\xi)=\sup_{f\in C_c(\Dl)}\biggl\{\int_\Dl f\,d\xi-\varphi(f)\biggr\}
\]
for any $\xi\in\cM(\Dl)$.
\end{enumerate}
\end{thm}

\subsection{Proofs of Theorems~\ref{thm:LLN_PD} and~\ref{thm:LDP_PD}}\label{ssec:pf_LLN-LDP_PD}
In this subsection, we apply Theorems~\ref{thm:LLN_PB-PD} and~\ref{thm:LDP_PBtuple-PD} to prove Theorems~\ref{thm:LLN_PD} and~\ref{thm:LDP_PD}, respectively.
\begin{proof}[Proof of Theorem~\ref{thm:LLN_PD}]
Note first that $\xi_q^{\X^n}([0,s]\times(t,\infty])=\b_q^{\X^n}(s,t)$ for every $0\le s\le t<\infty$ from Theorem~\ref{thm:mass_quadrant}.
Therefore, Theorem~\ref{thm:LLN_PB} implies that for any $0\le s<t<\infty$,
\[
\lim_{n\to\infty}\frac{\E[\xi_q^{\X^n}([0,s]\times(t,\infty])]}{|\Lm^n|}=\widehat\b_q(s,t)
\]
and
\[
\frac{\xi_q^{\X^n}([0,s]\times(t,\infty])}{|\Lm^n|}\to\widehat\b_q(s,t)\quad\text{almost surely as $n\to\infty$.}
\]
Note also that $\E[\xi_q^{\X^n}]/|\Lm^n|\in\cM(\Dl)$.
Therefore, applying Theorem~\ref{thm:LLN_PB-PD} with $\xi^n=\xi_q^{\X^n}/|\Lm^n|$ and $a_n=|\Lm^n|$, we obtain the conclusion.
\end{proof}
\begin{proof}[Proof of Theorem~\ref{thm:LDP_PD}]
We fix a finite family $\cP=\{(s_i,t_i)\}_{i=1}^h$ with $0\le s_i<t_i<\infty$.
Since $\xi_q^{\X^n}([0,s]\times(t,\infty])=\b_q^{\X^n}(s,t)$ from Theorem~\ref{thm:mass_quadrant}, it follows from Theorem~\ref{thm:LDP_PBtuple} that the $\R^h$-valued process
\[
\biggl\{\biggl(\frac{\xi_q^{\X^n}([0,s_1]\times(t_1,\infty])}{|\Lm^n|},\frac{\xi_q^{\X^n}([0,s_2]\times(t_2,\infty])}{|\Lm^n|},\ldots,\frac{\xi_q^{\X^n}([0,s_h]\times(t_h,\infty])}{|\Lm^n|}\biggr)\biggr\}_{n\in\N}
\]
satisfies the LDP with speed $|\Lm^n|$ and a good convex rate function that has a unique zero point.
Furthermore, combining~\eqref{eq:total_mass} and Lemma~\ref{lem:no_cube-vol}, we have $\xi_q^{\X^n}(\Dl)\le\#\cK^d_q(\Lm^n)\le3^d|\Lm^n|$.
Therefore,~\eqref{eq:LDP_PBtuple-PD} with $\xi^n=\xi_q^{\X^n}/|\Lm^n|$ and $a_n=|\Lm^n|$ is satisfied for any $\lm>0$.
Consequently, Theorem~\ref{thm:LDP_PBtuple-PD} implies that the $\cM(\Dl)$-valued process $\{\xi_q^{\X^n}/|\Lm^n|\}_{n\in\N}$ satisfies the LDP with speed $|\Lm^n|$ and a good convex rate function $I_q$, defined in~\eqref{eq:LDP_PD}, that has a unique zero point.
Furthermore, the unique zero point of $I_q$ must be $\widehat\xi_q$ in Theorem~\ref{thm:LLN_PD}.
\end{proof}

\subsection{Proof of Theorem~\ref{thm:LDP_PBtuple-PD}}\label{ssec:pf_PBtuple-PD}
In this subsection, we will prove Theorem~\ref{thm:LDP_PBtuple-PD}.
We first introduce the notion of histogram of a given measure $\xi\in\cM(\Dl)$, which is useful for the proof of Theorem~\ref{thm:LDP_PBtuple-PD}.
For $l\in\N$, let $\cI_l$ be the set of all disjoint rectangular regions $I$ of the form either
\begin{align}
I&=\biggl[0,\frac1{2^{l+1}}\biggr]\times\biggl(\frac{j-1}{2^{l+1}},\frac j{2^{l+1}}\biggr]\text{ for $j\in\N$ with $3\le j\le l\cdot2^{l+1}$}\label{eq:left-closed_rect}
\shortintertext{or}
I&=\biggl(\frac{i-1}{2^{l+1}},\frac i{2^{l+1}}\biggr]\times\biggl(\frac{j-1}{2^{l+1}},\frac j{2^{l+1}}\biggr]\text{ for $(i,j)\in\N^2$ with $2\le i\le j\le l\cdot2^{l+1}$ and $j-i\ge2$.}\label{eq:left-open_rect}
\end{align}
Given $I\in\bigcup_{l=1}^\infty\cI_l$, we denote by $\text{LR}(I)$, $\text{UR}(I)$, $\text{UL}(I)$, and $\text{LL}(I)$ the lower-right, upper-right, upper-left, and lower-left corners of $I$, respectively.
Note that every $I\in\bigcup_{l=1}^\infty\cI_l$ is a relatively compact set in $\Dl$ since $\text{LR}(I)\in\Dl$.
Therefore, for any $\xi\in \cM(\Dl)$ and $I\in\cI_l$, it holds that $\xi(I)<\infty$.
Note also that
\begin{equation}\label{eq:cover_Dl}
\bigcup_{l=1}^\infty\bigcup_{I\in\cI_l}I=\Dl.
\end{equation}
Given $\xi\in\cM(\Dl)$ and $l\in\N$, we define the \textit{histogram of $\xi$ with fineness degree $l$} by
\[
\hist_l(\xi)\coloneqq(\xi(I))_{I\in\cI_l}\in\R^{\cI_l}.
\]
\begin{lem}\label{lem:LDP_PB-PH}
Let $\{\xi^n\}_{n\in\N}$ be an $\cM(\Dl)$-valued process, and let $l\in\N$ be fixed.
Assume that for any finite family $\cP=\{(s_i,t_i)\}_{i=1}^h$ with $0\le s_i<t_i<\infty$, the $\R^h$-valued process
\[
\{(\xi^n([0,s_1]\times(t_1,\infty]),\xi^n([0,s_2]\times(t_2,\infty]),\ldots,\xi^n([0,s_h]\times(t_h,\infty]))\}_{n\in\N}
\]
satisfies an LDP with speed $a_n$ and a good rate function $I_\cP\colon\R^h\to[0,\infty]$.
Then, the $\R^{\cI_l}$-valued process $\{\hist_l(\xi^n)\}_{n\in\N}$ satisfies an LDP with speed $a_n$ and a good rate function $I_l\colon\R^{\cI_l}\to[0,\infty]$.
Furthermore, the following statements hold.
\begin{enumerate}
\item If $I_\cP$ has a unique zero point for each finite family $\{(s_i,t_i)\}_{i=1}^h$ with $0\le s_i<t_i<\infty$, then so does $I_l$.
\item Suppose that $I_\cP$ is convex for each finite family $\{(s_i,t_i)\}_{i=1}^h$ with $0\le s_i<t_i<\infty$, and also that
\[
\sup_{n\in\N}a_n^{-1}\log\E\biggl[\exp\biggl(a_n\sum_{I\in\cI_l}\lm_I\xi^n(I)\biggr)\biggr]<\infty
\]
for any $\lm=(\lm_I)_{I\in\cI_l}\in\R^{\cI_l}$.
Then, for every $\lm=(\lm_I)_{I\in\cI_l}\in\R^{\cI_l}$, the limit
\[
\varphi_l(\lm)\coloneqq\lim_{n\to\infty}a_n^{-1}\log\E\biggl[\exp\biggl(a_n\sum_{I\in\cI_l}\lm_I\xi^n(I)\biggr)\biggr]
\]
exists in $\R$, and $I=\varphi_l^*$ holds.
\end{enumerate}
\end{lem}
\begin{proof}
Set $\cP_l=\bigcup_{I\in\cI_l}\{\text{LR}(I)\text{, }\text{UR}(I)\text{, }\text{UL}(I)\text{, }\text{LL}(I)\}$.
We define a linear map $F_l\colon\R^{\cP_l}\to\R^{\cI_l}$ by
\[
(F_l(\b))(I)
\coloneqq\begin{cases}
\b(\text{LR}(I))-\b(\text{\text{UR}(I)})    &\text{if $I$ is of the form~\eqref{eq:left-closed_rect},}\\
\b(\text{LR}(I))-\b(\text{\text{UR}(I)})+\b(\text{UL}(I))-\b(\text{LL}(I))    &\text{if $I$ is of the form~\eqref{eq:left-open_rect}}
\end{cases}
\]
for any $\b=(\b(p))_{p\in\cP_l}\in\R^{\cP_l}$ and $I\in\cI_l$.
Note that $F_l$ is continuous.
Now, for $p=(s,t)\in\R^2$ with $0\le s\le t<\infty$, we write $\b^n(p)=\xi^n([0,s]\times(t,\infty])$ for convenience.
Then, by the inclusion-exclusion principle, we have
\[
\hist_l(\xi^n)=F_l((\b^n(p))_{p\in\cP_l}).
\]
Since the $\R^{\cP_l}$-valued process $\{(\b^n(p))_{p\in\cP_l}\}_{n\in\N}$ satisfies the LDP with speed $a_n$ and a good rate function $I_{\cP_l}$ from the assumption, it follows from Theorem~\ref{thm:contraction} that the $\R^{\cI_l}$-valued process $\{\hist_l(\xi^n)\}_{n\in\N}$ also satisfies the LDP with speed $a_n$ and a good rate function $I_l\colon\R^{\cI_l}\to[0,\infty]$ defined by
\[
I_l(H)\coloneqq\inf_{\b\in F_l^{-1}(\{H\})}I_{\cP_l}(\b)
\]
for any $H\in\R^{\cI_l}$.
Statement~(1) follows immediately from Remark~\ref{rem:contraction}~(1).
Furthermore, combining Remark~\ref{rem:contraction}~(2) with Theorem~\ref{thm:DZ_4.5.10} (with speed $a_n$ instead of $|\Lm^n|$) yields Statement~(2).
\end{proof}

Next, we prove the following lemma using the technique of exponentially good approximation~(see Appendix~\ref{sec:trans_LDP}).
In what follows, for $\xi\in\cM(\Dl)$ and $f\in C_c(\Dl)$, we write $\xi f=\int_\Dl f\,d\xi\in\R$ for simplicity.
\begin{lem}\label{lem:fdim_distr}
Let $\{\xi^n\}_{n\in\N}$ be an $\cM(\Dl)$-valued process, and let $m\in\N$ and $f_1,f_2,\ldots,f_m\in C_c(\Dl)$ be fixed.
Assume that~\eqref{eq:LDP_PBtuple-PD} holds for any $\lm>0$.
Assume further that for each fixed $l\in\N$, the $\R^{\cI_l}$-valued process $\{\hist_l(\xi^n)\}_{n\in\N}$ satisfies an LDP with speed $a_n$ and a good rate function $I_l\colon\R^{\cI_l}\to[0,\infty]$.
Then, the $\R^m$-valued process $\{(\xi^nf_1,\xi^nf_2,\ldots,\xi^nf_m)\}_{n\in\N}$ satisfies an LDP with speed $a_n$ and a good rate function $I_{f_1,f_2,\ldots,f_m}\colon\R^m\to[0,\infty]$.
Furthermore, the following statements hold.
\begin{enumerate}
\item Suppose that there exists $K>0$ such that $\xi^n(\Dl)\le K$ for all $n\in\N$.
If $I_l$ has a unique zero point for each $l\in\N$, then so does $I_{f_1,f_2,\ldots,f_m}$.
\item If $I_l$ is convex for each $l\in\N$, then for every $\lm=(\lm_j)_{j=1}^m\in\R^m$, the limit
\[
\varphi_{f_1,f_2,\ldots,f_m}(\lm)\coloneqq\lim_{n\to\infty}a_n^{-1}\log\E\biggl[\exp\biggl(a_n\sum_{j=1}^m\lm_j(\xi^nf_j)\biggr)\biggr]
\]
exists in $\R$, and $I=\varphi_{f_1,f_2,\ldots,f_m}^*$ holds.
\end{enumerate}
\end{lem}
\begin{proof}
We first note that~\eqref{eq:LDP_PBtuple-PD} implies that $\{\xi^n(\Dl)\}_{n\in\N}$ is exponentially tight with speed $a_n$.
Indeed, for any $K\ge0$, the Markov inequality after multiplying $a_n$ and exponentiating yields
\[
a_n^{-1}\log\P(\xi^n(\Dl)>K)
\le a_n^{-1}\log\frac{\E[\exp(a_n\xi^n(\Dl))]}{\exp(a_nK)}
=a_n^{-1}\log\E[\exp(a_n\xi^n(\Dl))]-K,
\]
which immediately implies the exponential tightness of $\{\xi^n(\Dl)\}_{n\in\N}$ with speed $a_n$.

Now, for $i\in\{1,2,\ldots,m\}$ and $l\in\N$, define a piecewise constant function $f_i^{(l)}\colon\Dl\to\R$ by
\[
f_i^{(l)}\coloneqq\sum_{I\in\cI_l}f_i(\text{UR}(I))\1_I.
\]
For each $n\in\N$ and $l\in\N$, we set
\[
Z^n=(\xi^nf_1,\xi^nf_2,\ldots,\xi^nf_m)
\quad\text{and}\quad
Z^{n,l}=(\xi^nf_1^{(l)},\xi^nf_2^{(l)},\ldots,\xi^nf_m^{(l)}).
\]
In order to apply Theorem~\ref{thm:EGA}, we first show that $\{Z^{n,l}\}_{n,l\in\N}$ is an exponentially good approximation of $\{Z^n\}_{n\in\N}$ with speed $a_n$.
Let $\dl>0$ and $\a>0$.
By the exponential tightness of $\{\xi^n(\Dl)\}_{n\in\N}$ with speed $a_n$, there exists a constant $K>0$ such that
\begin{equation}\label{eq:fdim_distr_1}
\limsup_{n\to\infty}a_n^{-1}\log\P(\xi^n(\Dl)>K)\le-\a.
\end{equation}
Since $f_1,f_2,\ldots,f_m\in C_c(\Dl)$ are uniformly continuous, we can take $\rho>0$ such that
\[
|f_i(x)-f_i(y)|\le\frac\dl{\sqrt mK}
\]
for any $i\in\{1,2,\ldots,m\}$ and $x,y\in\Dl$ with $d_\Dl(x,y)<\rho$.
Here, $d_\Dl$ is a metric that induces the topology on $\Dl$.
Noting that~\eqref{eq:cover_Dl} and the compactness of $\bigcup_{i=1}^m\supp(f_i)$, choose $L\in\N$ so that
\[
\bigcup_{i=1}^m\supp(f_i)\subset\bigcup_{I\in\cI_L}I
\quad\text{and}\quad
\max_{\substack{I\in\cI_L\\I\cap\bigcup_{i=1}^m\supp(f_i)\neq\emptyset}}\diam_\Dl(I)
<\rho.
\]
Then, for $l\ge L$ and $i\in\{1,2,\ldots,m\}$,
\begin{align*}
\sup_{x\in\Dl}|f_i(x)-f_i^{(l)}(x)|
&=\sup_{x\in\bigcup_{I\in\cI_L}I}|f_i(x)-f_i^{(l)}(x)|\\
&=\max_{\substack{I\in\cI_L\\I\cap\bigcup_{i=1}^m\supp(f_i)\neq\emptyset}}\max_{\substack{J\in\cI_l\\J\subset I}}\sup_{x\in J}|f_i(x)-f_i^{(l)}(x)|\\
&=\max_{\substack{I\in\cI_L\\I\cap\bigcup_{i=1}^m\supp(f_i)\neq\emptyset}}\max_{\substack{J\in\cI_l\\J\subset I}}\sup_{x\in J}|f_i(x)-f_i(\text{UR}(J))|\\
&\le\frac\dl{\sqrt mK}.
\end{align*}
Therefore, for $l\ge L$,
\[
\|Z^n-Z^{n,l}\|_{\R^m}
\le\Biggl(\sum_{i=1}^m\biggl(\int_\Dl|f_i-f_i^{(l)}|\,d\xi^n\biggr)^2\Biggr)^{1/2}
\le\frac{\xi^n(\Dl)}K\dl,
\]
which implies that
\[
\limsup_{n\to\infty}a_n^{-1}\log\P(\|Z^n-Z^{n,l}\|_{\R^m}>\dl)
\le\limsup_{n\to\infty}a_n^{-1}\log\P(\xi^n(\Dl)>K)
\le-\a.
\]
Thus, we obtain
\[
\limsup_{l\to\infty}\limsup_{n\to\infty}a_n^{-1}\log\P(\|Z^n-Z^{n,l}\|_{\R^m}>\dl)\le-\a.
\]
Since $\a>0$ is arbitrary,
\begin{equation}\label{eq:fdim_distr_2}
\limsup_{l\to\infty}\limsup_{n\to\infty}a_n^{-1}\log\P(\|Z^n-Z^{n,l}\|_{\R^m}>\dl)=-\infty,
\end{equation}
which means that $\{Z^{n,l}\}_{n,l\in\N}$ is an exponentially good approximation of $\{Z^n\}_{n\in\N}$ with speed $a_n$.

Next, we fix $l\in\N$, and prove that the $\R^m$-valued process $\{Z^{n,l}\}_{n\in\N}$ satisfies an LDP with speed $a_n$ and a rate function.
For each $i\in\{1,2,\ldots,m\}$, we define a linear map $G_i^{(l)}\colon\R^{\cI_l}\to\R$ by
\[
G_i^{(l)}(H)\coloneqq\sum_{I\in\cI_l}f_i(\text{UR}(I))H(I)
\]
for any $H=(H(I))_{I\in\cI_l}\in\R^{\cI_l}$, and also define a linear map $G^{(l)}\colon\R^{\cI_l}\to\R^m$ by
\[
G^{(l)}\coloneqq(G_1^{(l)},G_2^{(l)},\ldots,G_m^{(l)}).
\]
Since
\[
G_i^{(l)}(\hist_l(\xi^n))
=\sum_{I\in\cI_l}f_i(\text{UR}(I))\xi^n(I)
=\xi^nf_i^{(l)}
\]
for all $i\in\{1,2,\ldots,m\}$, we have $G^{(l)}(\hist_l(\xi^n))=Z^{n,l}$.
Therefore, it follows from the assumption and Theorem~\ref{thm:contraction} that $\{Z^{n,l}\}_{n\in\N}$ satisfies an LDP with speed $a_n$ and a good rate function.

Furthermore, since $|\xi^nf_i|\le(\sup_{x\in\Dl}|f_i(x)|)\xi^n(\Dl)$, the exponential tightness of $\{Z^n\}_{n\in\N}$ with speed $a_n$ follows immediately from that of $\{\xi^n(\Dl)\}_{n\in\N}$.
Thus, by Theorem~\ref{thm:EGA}, the $\R^m$-valued process $\{Z^n\}_{n\in\N}$ satisfies an LDP with speed $a_n$ and a good rate function.

For Statement~(1), suppose that there exists $K>0$ such that $\xi^n(\Dl)\le K$ for all $n\in\N$.
Then, we can replace $\a$ in~\eqref{eq:fdim_distr_1} to $\infty$.
Hence, instead of~\eqref{eq:fdim_distr_2}, we obtain
\[
\limsup_{n\to\infty}a_n^{-1}\log\P(\|Z^n-Z^{n,l}\|_{\R^m}>\dl)=-\infty
\]
for any $l\ge L$.
Therefore, the conclusion follows from Remark~\ref{rem:EGA}~(1).

Lastly, we prove Statement~(2).
By the Cauchy--Schwarz inequality,
\begin{align*}
\sup_{n\in\N}a_n^{-1}\log\E\biggl[\exp\biggl(a_n\sum_{j=1}^m\lm_j(\xi^nf_j)\biggr)\biggr]
&\le\sup_{n\in\N}a_n^{-1}\log\E\biggl[\exp\biggl(a_n\|\lm\|_{\R^m}\biggl(\sum_{j=1}^m(\xi^nf_j)^2\biggr)^{1/2}\biggr)\biggr]\\
&\le\sup_{n\in\N}a_n^{-1}\log\E\biggl[\exp\biggl(a_n\|\lm\|_{\R^m}\biggl(\sum_{j=1}^m\sup_{x\in\Dl}f_j(x)^2\biggr)^{1/2}\xi^n(\Dl)\biggr)\biggr]
\end{align*}
for $\lm=(\lm_j)_{j=1}^m\in\R^m$.
The right-hand side of the above inequality is finite by~\eqref{eq:LDP_PBtuple-PD}.
Thus, combining Remark~\ref{rem:EGA}~(2) with Theorem~\ref{thm:DZ_4.5.10} (with speed $a_n$ instead of $|\Lm^n|$) yields Statement~(2).
\end{proof}

Finally, we prove Theorem~\ref{thm:LDP_PBtuple-PD} using Theorem~\ref{thm:LDP_rdm_meas}, which is a general statement to ensure an LDP for random measures (see Appendix~\ref{sec:LDP_rdm_meas}).
\begin{proof}[Proof of Theorem~\ref{thm:LDP_PBtuple-PD}]
Combining Lemmas~\ref{lem:LDP_PB-PH} and~\ref{lem:fdim_distr}, we conclude that for any $m\in\N$ and $f_1,f_2,\ldots,f_m\in C_c(\Dl)$, the $\R^m$-valued process $\{(\xi^nf_1,\xi^nf_2,\ldots,\xi^nf_m)\}_{n\in\N}$ satisfies an LDP with speed $a_n$ and a good rate function.
Therefore, Theorem~\ref{thm:LDP_rdm_meas} implies that the $\cM(\Dl)$-valued process $\{\xi^n\}_{n\in\N}$ satisfies an LDP with speed $a_n$ and a good rate function.

For Statement~(1), suppose that there exists $K>0$ such that $\xi^n(\Dl)\le K$ for all $n\in\N$.
Then, the conclusion is an immediate consequence of combining Statements~(1) of Lemmas~\ref{lem:LDP_PB-PH} and~\ref{lem:fdim_distr} and Theorem~\ref{thm:LDP_rdm_meas}.

Lastly, we prove Statement~(2).
We first note that from~\eqref{eq:LDP_PBtuple-PD},
\begin{align*}
\sup_{n\in\N}a_n^{-1}\log\E\biggl[\exp\biggl(a_n\sum_{I\in\cI_l}\lm_I\xi^n(I)\biggr)\biggr]
&\le\sup_{n\in\N}a_n^{-1}\log\E\Bigl[\exp\Bigl(a_n\max_{I\in\cI_l}|\lm_I|\cdot\xi^n(\Dl)\Bigr)\Bigr]<\infty
\end{align*}
for any $\lm=(\lm_I)_{I\in\cI_l}\in\R^{\cI_l}$, and 
\[
\sup_{n\in\N}a_n^{-1}\log\E[\exp(a_n\xi^nf)]
\le\sup_{n\in\N}a_n^{-1}\log\E\Bigl[\exp\Bigl(a_n\sup_{x\in\Dl}|f(x)|\cdot\xi^n(\Dl)\Bigr)\Bigr]
<\infty
\]
for any $f\in C_c(\Dl)$.
Therefore, Statement~(2) follows immediately by combining Statements~(2) of Lemmas~\ref{lem:LDP_PB-PH} and~\ref{lem:fdim_distr} and Theorem~\ref{thm:LDP_rdm_meas}.
\end{proof}

\appendix

\section{Strong law of large numbers for strongly regular nearly additive processes}\label{sec:LLN_rnap}
In this section, we prove the strong law of large numbers under a weaker assumption than that in Section~\ref{sec:LDP_ernap}.
In what follows, let $h\in\N$ be fixed.
\begin{df}
Let $r\ge0$ be an integer.
We say that an $\R^h$-valued process $\{S^n\}_{n\in\N}$ is \textit{strongly $r$-nearly additive} if there exist $\R^h$-valued random variables $\{S^{n,z}\}_{n\in\N,z\in\Z^d}$ such that the following conditions are satisfied:
\begin{itemize}
\item $\{S^{n,z}\}_{z\in\Z^d}$ are independent copies of $S^n$ for every $n\in\N$;
\item it holds that
\begin{equation}\label{eq:sna}
\sup_{m\in\N}|\Lm^{(2m+1)k}|^{-1}\biggl\|S^{(2m+1)k}-\sum_{z\in\Z^d\cap[-m,m]^d}S^{k-r,z}\biggr\|_{\R^h}
\xrightarrow[k\to\infty]{\text{a.s.}}0.
\end{equation}
\end{itemize}
We also say that an $\R^h$-valued process $\{S^n\}_{n\in\N}$ is \textit{strongly nearly additive} if there exists an integer $r\ge0$ such that $\{S^n\}_{n\in\N}$ is strongly $r$-nearly additive.
\end{df}
\begin{df}
We say that an $\R^h$-valued process $\{S^n\}_{n\in\N}$ is \textit{strongly regular} if the following property holds for each fixed $k\in\N$:
if $m_n$ is taken as the unique integer satisfying that $(2m_n+1)k\le n<(2m_n+3)k$ for each $n\in\N$, then
\begin{equation}\label{eq:sr}
|\Lm^n|^{-1}\|S^n-S^{(2m_n+1)k}\|_{\R^h}
\xrightarrow[n\to\infty]{\text{a.s.}}0.
\end{equation}
\end{df}
\begin{rem}\label{rem:Exp-STG}
By a standard Borel--Cantelli argument, we can easily verify that the exponential regularity and exponentially near additivity implies the strong regularity and strongly near additivity, respectively.
\end{rem}
\begin{thm}
Let $\{S^n\}_{n\in\N}$ be a strongly regular nearly additive $\R^h$-valued process consisting of integrable random variables.
Suppose that $\sup_{n\in\N}\|\E[S^n]\|_{\R^h}/|\Lm^n|<\infty$.
Then, the limit
\begin{equation}\label{eq:sLLN_srnap}
\widehat S\coloneqq\lim_{n\to\infty}\frac{\E[S^n]}{|\Lm^n|}
\end{equation}
exists in $\R^h$, and
\[
\frac{S^n}{|\Lm^n|}\to\widehat S\quad\text{almost surely as $n\to\infty$.}
\]
\end{thm}
\begin{proof}
Since $\sup_{n\in\N}\|\E[S^n]\|_{\R^h}/|\Lm^n|<\infty$ from the assumption, there exists an accumulation point $\widehat S\in\R^h$ of the sequence $\{\E[S^n]/|\Lm^n|\}_{n\in\N}$.
From the strongly near additivity of $\{S^n\}_{n\in\N}$, we can take an integer $r\ge0$ and $\R^h$-valued random variables $\{S^{n,z}\}_{n\in\N,z\in\Z^d}$ such that $\{S^{n,z}\}_{z\in\Z^d}$ are independent copies of $S^n$ for every $n\in\N$ and
\begin{equation}\label{eq:sLLN_srnap_1}
\sup_{m\in\N}\frac1{|\Lm^{(2m+1)k}|}\biggl\|S^{(2m+1)k}-\sum_{z\in\Z^d\cap[-m,m]^d}S^{k-r,z}\biggr\|_{\R^h}
\xrightarrow[k\to\infty]{\text{a.s.}}0.
\end{equation}
Now, let $k\in\N$ be fixed, and let $m_n$ be the unique integer satisfying that $(2m_n+1)k\le n<(2m_n+3)k$ for each $n\in\N$.
Then, by the triangle inequality,
\begin{align}\label{eq:sLLN_srnap_2}
\biggl\|\frac{S^n}{|\Lm^n|}-\widehat S\biggr\|_{\R^h}
&\le\frac1{|\Lm^n|}\|S^n-S^{(2m_n+1)k}\|_{\R^h}\nonumber\\
&\qad+\frac{|\Lm^{(2m_n+1)k}|}{|\Lm^n|}\cdot\frac1{|\Lm^{(2m_n+1)k}|}\biggl\|S^{(2m_n+1)k}-\sum_{z\in\Z^d\cap[-m_n,m_n]^d}S^{k-r,z}\biggr\|_{\R^h}\nonumber\\
&\qad+\frac{(2m_n+1)^d|\Lm^{k-r}|}{|\Lm^n|}\biggl\|\frac1{(2m_n+1)^d}\sum_{z\in\Z^d\cap[-m_n,m_n]^d}\frac{S^{k-r,z}}{|\Lm^{k-r}|}-\frac{\E[S^{k-r}]}{|\Lm^{k-r}|}\biggr\|_{\R^h}\nonumber\\
&\qad+\frac{(2m_n+1)^d|\Lm^{k-r}|}{|\Lm^n|}\biggl\|\frac{\E[S^{k-r}]}{|\Lm^{k-r}|}-\widehat S\biggr\|_{\R^h}
+\biggl(1-\frac{(2m_n+1)^d|\Lm^{k-r}|}{|\Lm^n|}\biggr)\|\widehat S\|_{\R^h}\nonumber\\
&\le\frac1{|\Lm^n|}\|S^n-S^{(2m_n+1)k}\|_{\R^h}\nonumber\\
&\qad+\sup_{m\in\N}\frac1{|\Lm^{(2m+1)k}|}\biggl\|S^{(2m+1)k}-\sum_{z\in\Z^d\cap[-m,m]^d}S^{k-r,z}\biggr\|_{\R^h}\nonumber\\
&\qad+\biggl\|\frac1{(2m_n+1)^d}\sum_{z\in\Z^d\cap[-m_n,m_n]^d}\frac{S^{k-r,z}}{|\Lm^{k-r}|}-\frac{\E[S^{k-r}]}{|\Lm^{k-r}|}\biggr\|_{\R^h}\nonumber\\
&\qad+\biggl\|\frac{\E[S^{k-r}]}{|\Lm^{k-r}|}-\widehat S\biggr\|_{\R^h}
+\biggl(1-\frac{(2m_n+1)^d|\Lm^{k-r}|}{|\Lm^n|}\biggr)\|\widehat S\|_{\R^h}
\end{align}
for any $n\in\N$.
For the second inequality, we also used $(2m_n+1)^d|\Lm^{k-r}|\le|\Lm^{(2m_n+1)k}|\le|\Lm^n|$.
The first and third terms in the right-hand side of~\eqref{eq:sLLN_srnap_2} converges to zero almost surely as $n\to\infty$ because of the strongly regularity of $\{S^n\}_{n\in\N}$ and the strong law of large numbers, respectively.
Noting also that
\[
\lim_{n\to\infty}\frac{(2m_n+1)^d|\Lm^{k-r}|}{|\Lm^n|}=\biggl(1-\frac rk\biggr)^d,
\]
we take $n\to\infty$ of both sides of~\eqref{eq:sLLN_srnap_2} to obtain
\begin{align*}
\limsup_{n\to\infty}\biggl\|\frac{S^n}{|\Lm^n|}-\widehat S\biggr\|_{\R^h}
&\le\sup_{m\in\N}\frac1{|\Lm^{(2m+1)k}|}\biggl\|S^{(2m+1)k}-\sum_{z\in\Z^d\cap[-m,m]^d}S^{k-r,z}\biggr\|_{\R^h}\\
&\qad+\biggl\|\frac{\E[S^{k-r}]}{|\Lm^{k-r}|}-\widehat S\biggr\|_{\R^h}
+\biggl(1-\biggl(1-\frac rk\biggr)^d\biggr)\|\widehat S\|_{\R^h}.
\end{align*}
By letting $k\to\infty$ in the above inequality (choose a suitable subsequence in $k$ if necessary), we conclude from~\eqref{eq:sLLN_srnap_1} that $S^n/|\Lm^n|$ converges to $\widehat S$ almost surely as $n\to\infty$.
In particular, $\widehat S$ is a unique accumulation point of the sequence $\{\E[S^n]/|\Lm^n|\}_{n\in\N}$.
Finally, the uniqueness of the accumulation point $\widehat S$ together with $\sup_{n\in\N}\|\E[S^n]\|_{\R^h}/|\Lm^n|<\infty$ implies~\eqref{eq:sLLN_srnap}.
\end{proof}

\begin{rem}
The \textit{weakly near additivity} and \textit{weak regularity} of an $\R^h$-valued process $\{S^n\}_{n\in\N}$ are defined by replacing the almost sure convergence in~\eqref{eq:sna} and~\eqref{eq:sr} to the convergence in probability.
For a weakly regular nearly additive $\R^h$-valued process $\{S^n\}_{n\in\N}$ consisting of integrable random variables with $\sup_{n\in\N}\|\E[S^n]\|_{\R^h}/|\Lm^n|<\infty$, the weak law of large numbers holds: the limit
\[
\widehat S\coloneqq\lim_{n\to\infty}\frac{\E[S^n]}{|\Lm^n|}
\]
exists in $\R^h$, and
\[
\frac{S^n}{|\Lm^n|}\to\widehat S\quad\text{in probability as $n\to\infty$.}
\]
Since the proof is almost the same, we omit the proof.
\end{rem}

\section{Transformations of large deviation principles}\label{sec:trans_LDP}
We here review basic methods to move around LDPs between different spaces.
In the following, let $\{a_n\}_{n\in\N}$ be a sequence of positive numbers tending to infinity as $n\to\infty$.

We first state the \textit{contraction principle}, which states that an LDP is preserved under continuous maps.
\begin{thm}[{\cite[Theorem~4.2.1]{DZ10}}]\label{thm:contraction}
Let $\cX$ and $\cY$ be Hausdorff topological spaces, and let $F\colon\cX\to\cY$ be a continuous function.
Assume that an $\cX$-valued process $\{Z^n\}_{n\in\N}$ satisfies an LDP with speed $a_n$ and a good rate function $I\colon\cX\to[0,\infty]$.
Then, the $\cY$-valued process $\{F(Z^n)\}_{n\in\N}$ satisfies the LDP with speed $a_n$ and a good rate function $I'\colon\cY\to[0,\infty]$ defined by
\begin{equation}\label{eq:contraction}
I'(y)\coloneqq\inf_{x\in F^{-1}(\{y\})}I(x)
\end{equation}
for any $y\in\cY$.
Here, the infimum over the empty set is regarded as $\infty$ by convention.
\end{thm}
\begin{rem}\label{rem:contraction}
We remark on the uniqueness of the zero point and the convexity of the rate function $I'$ in Theorem~\ref{thm:contraction}.
\begin{enumerate}
\item (Uniqueness of the zero point)~If the good rate function $I$ has a unique zero point, then so does $I'$.
Indeed, let $\widehat x\in\cX$ be the unique zero point of $I$.
Then, $I'(F(\widehat x))=0$ from the definition of $I'$.
Furthermore, if $I'(y)=0$, then
there exists $x\in F^{-1}(\{y\})$ such that $I(x)=0$ by the goodness of the rate function $I$ together with the continuity of $F$.
Thus, $x=\widehat x$ from the uniqueness of the zero point of $I$, hence necessarily $y=F(\widehat x)$.
\item (Convexity)~Suppose that $\cX=\R^h$ and $\cY=\R^{h'}$ for some $h,h'\in\N$ and that $f$ is a linear map.
We can easily verify from~\eqref{eq:contraction} that if the rate function $I$ is convex, then so is $I'$.
\end{enumerate}
\end{rem}

Next, we review the notion of exponentially good approximation, and state a technical result, which deduces a new LDP from LDPs for approximation sequences.
\begin{df}
Let $(\cY,d_\cY)$ be a metric space, and let $\{Z^n\}_{n\in\N}$ and $\{Z^{n,l}\}_{n,l\in\N}$ be $\cY$-valued random variables.
$\{Z^{n,l}\}_{n,l\in\N}$ is called an \textit{exponentially good approximation} of $\{Z^n\}_{n\in\N}$ with speed $a_n$ if for any $\dl>0$,
\[
\lim_{l\to\infty}\limsup_{n\to\infty}a_n^{-1}\log\P(d_\cY(Z^n,Z^{n,l})>\dl)=-\infty.
\]
\end{df}
Combining~\cite[Lemma~1.2.18]{DZ10} and~\cite[Theorem~4.2.16]{DZ10}, we immediately have the following.
\begin{thm}\label{thm:EGA}
Let $(\cY,d_\cY)$ be a metric space, and let $\{Z^n\}_{n\in\N}$ and $\{Z^{n,l}\}_{n,l\in\N}$ be $\cY$-valued random variables.
Assume that the following three conditions are satisfied.
\begin{itemize}
\item $\{Z^{n,l}\}_{n,l\in\N}$ is an exponentially good approximation of $\{Z^n\}_{n\in\N}$ with speed $a_n$.
\item For each fixed $l\in\N$, the $\cY$-valued process $\{Z^{n,l}\}_{n\in\N}$ satisfies an LDP with speed $a_n$ and a rate function $I_l\colon\cY\to[0,\infty]$.
\item $\{Z^n\}_{n\in\N}$ is exponentially tight with speed $a_n\colon$ for any $\a>0$, there exists a compact set $K\subset\cY$ such that
\[
\limsup_{n\to\infty}a_n^{-1}\log\P(Z^n\notin K)\le-\a.
\]
\end{itemize}
Then, the $\cY$-valued process $\{Z^n\}_{n\in\N}$ satisfies the LDP with speed $a_n$ and a good rate function $I\colon\cY\to[0,\infty]$ defined by
\begin{equation}\label{eq:EGA_rate_1}
I(y)\coloneqq\sup_{\dl>0}\limsup_{l\to\infty}\inf_{y'\in \bar B(y,\dl)}I_l(y')
\end{equation}
for any $y\in\cY$.
\end{thm}
\begin{rem}\label{rem:EGA}
We remark on the uniqueness of the zero point and the convexity of the rate function $I$ in Theorem~\ref{thm:EGA}.
\begin{enumerate}
\item (Uniqueness of the zero point)~Suppose that $\{Z^{n,l}\}_{n,l\in\N}$ satisfies the following slightly stronger condition than the usual exponentially good approximation condition: for any $\dl>0$, there exists $L_\dl\in\N$ such that
\[
\limsup_{n\to\infty}a_n^{-1}\log\P(\|Z^n-Z^{n,l}\|_{\R^m}>\dl)=-\infty
\]
for any $l\ge L_\dl$.
In this case, by a simple modification of the proof of Theorem~\ref{thm:EGA}, the rate function $I$ is given by
\begin{equation}\label{eq:EGA_rate_2}
I(y)=\sup_{\dl>0}\sup_{l\ge L_\dl}\inf_{y'\in \bar B(y,\dl)}I_l(y').
\end{equation}
From~\eqref{eq:EGA_rate_2}, it is not hard to verify that if $I_l$ has a unique zero point for each $l\in\N$, then so does $I$.
In fact, the unique zero point of $I$ is given as the limit of the unique zero point of $I_l$ with respect to $l$.
\item (Convexity)~Suppose that $\cX=\R^h$ and $\cY=\R^{h'}$ for some $h,h'\in\N$.
Then,~\eqref{eq:EGA_rate_1} implies that if the rate function $I_l$ is convex for each $l\in\N$, then so is $I$.
\end{enumerate}
\end{rem}

\section{Large deviation principle for random measures}\label{sec:LDP_rdm_meas}
In this section, we provide a sufficient condition for an LDP for random measures.
Theorem~\ref{thm:LDP_rdm_meas} is the main statement in this section, which will be used for proving the LDP for persistence diagrams in Section~\ref{sec:LDP_PDs}.

Let $\Dl$ be a (general) locally compact Hausdorff space with countable base, hence necessarily $\Dl$ is a complete and separable metric space, the so-called Polish space.
Let $C_c(\Dl)$ be the set of all real-valued continuous functions on $\Dl$ with compact support.
A Borel measure $\xi$ on $\Dl$ is called a \textit{Radon measure} if $\xi(K)<\infty$ for every compact set $K\subset\Dl$.
Let $\cM(\Dl)$ denote the set of all Radon measures on $\Dl$.
We equip $\cM(\Dl)$ with the vague topology, i.e., the weakest topology such that for every $f\in C_c(\Dl)$, the map $\cM(\Dl)\ni\xi\mapsto\xi f\coloneqq\int_\Dl f\,d\xi\in\R$ is continuous.
In fact, $\cM(\Dl)$ with the vague topology is a Polish space~(see, e.g., Lemma~4.6 in~\cite{Kal17}).

Given a Radon measure $\xi\in\cM(\Dl)$, we define a linear functional $L_\xi$ on $C_c(\Dl)$ by $L_\xi(f)\coloneqq\xi f$ for any $f\in C_c(\Dl)$.
Note that $L_\xi$ is positive in the sense that $L_\xi(f)\ge0$ for any $f\in C_c(\Dl)$ with $f\ge0$.
It is well known that there exists a one-to-one correspondence between $\cM(\Dl)$ and the set of all positive linear functionals on $C_c(\Dl)$, and the correspondence is given by $\xi\mapsto L_\xi$.
\begin{thm}[Riesz--Markov--Kakutani representation theorem]\label{thm:RMK}
Let $\Dl$ be a locally compact Hausdorff space with countable base.
Then, for any positive linear functional $L$ on $C_c(\Dl)$, there exists a unique Radon measure $\xi$ on $\Dl$ such that $L(f)=\xi f\coloneqq\int_\Dl f\,d\xi$ for any $f\in C_c(\Dl)$.
\end{thm}
\begin{rem}\label{rem:RMK}
Let $C_c(\Dl)'$ be the set of all linear functionals on $C_c(\Dl)$, and let $C_c(\Dl)'_+$ denote the subset of $C_c(\Dl)'$ consisting of all positive linear functionals.
It follows from Theorem~\ref{thm:RMK} and the preceding discussion that the map $\Phi\colon\cM(\Dl)\ni\xi\to L_\xi\in C_c(\Dl)'_+$ is bijective.
We equip $C_c(\Dl)'$ with the weak-* topology, i.e., the weakest topology such that for every $f\in C_c(\Dl)$, the evaluation map $\pi_f\colon C_c(\Dl)'\ni L\mapsto L(f)\in\R$ is continuous.
In other words, the weak-* topology on $C_c(\Dl)'$ is generated by all the sets of the form $\pi_f^{-1}(A)$ for some $f\in C_c(\Dl)$ and open set $A\subset\R$.
Hence, we can easily verify that $C_c(\Dl)'_+$ is a closed set of $C_c(\Dl)'$ and that the map $\Phi\colon\cM(\Dl)\to C_c(\Dl)'_+$ is homeomorphism with respect to the vague topology on $\cM(\Dl)$ and the relative topology on $C_c(\Dl)'_+$ induced from the weak-* topology on $C_c(\Dl)'$.
\end{rem}

Our aim in this section is to prove the following theorem.
In what follows, let $\{a_n\}_{n\in\N}$ be a sequence of positive numbers tending to infinity as $n\to\infty$.
\begin{thm}\label{thm:LDP_rdm_meas}
Let $\Dl$ be a locally compact Hausdorff space with countable base.
Let $\{\xi^n\}_{n\in\N}$ be an $\cM(\Dl)$-valued process.
Assume that for any $m\in\N$ and $f_1,f_2,\ldots,f_m\in C_c(\Dl)$, the $\R^m$-valued process $\{(\xi^nf_1,\xi^nf_2,\ldots,\xi^nf_m)\}_{n\in\N}$ satisfies an LDP with speed $a_n$ and a good rate function $I_{f_1,f_2,\ldots,f_m}\colon\R^m\to[0,\infty]$.
Then, the $\cM(\Dl)$-valued process $\{\xi^n\}_{n\in\N}$ satisfies the LDP with speed $a_n$ and a good rate function $I\colon\cM(\Dl)\to[0,\infty]$ defined by
\begin{equation}\label{eq:LDP_rdm_meas}
I(\xi)\coloneqq\sup_{m\in\N}\sup_{f_1,f_2,\ldots,f_m\in C_c(\Dl)}I_{f_1,f_2,\ldots,f_m}((\xi f_1,\xi f_2,\ldots,\xi f_m))
\end{equation}
for any $\xi\in\cM(\Dl)$.
Furthermore, the following statements hold.
\begin{enumerate}
\item If $I_f$ has a unique zero point for every $f\in C_c(\Dl)$, then so does $I$.
\item Suppose that $I_{f_1,f_2,\ldots,f_m}$ is convex for any $m\in\N$ and $f_1,f_2,\ldots,f_m\in C_c(\Dl)$, and also that
\[
\sup_{n\in\N}a_n^{-1}\log\E[\exp(a_n\xi^nf)]<\infty
\]
for any $f\in C_c(\Dl)$.
Then, for every $f\in C_c(\Dl)$, the limit
\[
\varphi(f)\coloneqq\lim_{n\to\infty}a_n^{-1}\log\E[\exp(\xi^nf)]
\]
exists in $\R$, and it holds that
\begin{equation}\label{eq:LDP_rdm_meas_1}
I(\xi)=\sup_{f\in C_c(\Dl)}\{\xi f-\varphi(f)\}
\end{equation}
for any $\xi\in\cM(\Dl)$.
\end{enumerate}
\end{thm}

A key for the proof is the Dawson--G\"artner theorem, which is a useful tool to lift a collection of LDPs in relatively small spaces into an LDP in a larger space identified as their projective limit.
We begin by reviewing the notion of projective system and projective limit.
Let $(J,\le)$ be a partially ordered set.
Assume that for any $i,j\in J$, there exists $k\in J$ such that both $i\le k$ and $j\le k$ hold.
Let $\{\cY_j\}_{j\in J}$ be a family of Hausdorff spaces, and let $\{p_{ij}\colon\cY_j\to\cY_i\}_{i\le j\in J}$ be a family of continuous maps satisfying that $p_{ij}\circ p_{jk}=p_{ik}$ for any $i\le j\le k$ and that $p_{jj}$ is the identity map on $\cY_j$ for any $j\in J$.
A pair $(\cY_j,p_{ij})_{i\le j\in J}$ is called a \textit{projective system}.
The \textit{projective limit} of a projective system $(\cY_j,p_{ij})_{i\le j\in J}$ is defined as
\[
\widetilde\cY=\plim[j\in J]\cY_j
\coloneqq\left\{x=(y_j)_{j\in J}\in\prod_{j\in J}\cY_j\relmiddle|y_i=p_{ij}(y_j)\text{ for any }i\le j\in J\right\},
\]
equipped with the relative topology induced from the product topology of $\prod_{j\in J}\cY_j$.
For each $j\in J$, let $p_j\colon\widetilde\cY\to\cY_j$ denote the canonical projection that maps $(y_j)_{j\in J}\in\widetilde\cY$ to $y_j\in\cY_j$.

\begin{thm}[Dawson--G\"artner theorem]\label{thm:DG}
Let $\widetilde\cY$ be the projective limit of a projective system $(\cY_j,p_{ij})_{i\le j\in J}$, and let $\{S^n\}_{n\in\N}$ be a $\widetilde\cY$-valued process.
Suppose that for every $j\in J$, the $\cY_j$-valued process $\{p_j(S^n)\}_{n\in\N}$ satisfies the LDP with speed $a_n$ and a good rate function $I_j\colon\cY_j\to[0,\infty]$.
Then, the $\widetilde\cY$-valued process $\{S^n\}_{n\in\N}$ satisfies the LDP with speed $a_n$ and good rate function $I\colon\widetilde\cY\to[0,\infty]$ defined by
\[
I(x)\coloneqq\sup_{j\in J}I_j(p_j(x))
\]
for any $x\in\widetilde\cY$.
\end{thm}

An important application of Theorem~\ref{thm:DG} is the case where the projective limit is identified as an algebraic dual, equipped with the weak-* topology, of an infinite dimensional real vector space as follows.
Given a real vector space $W$, let $W'$ denote its \textit{algebraic dual}, i.e., the set of all linear functionals on $W$.
We define a topological space $\cX$ as the algebraic dual $W'$ with the weak-* topology, i.e., the weakest topology such that for every $w\in W$, the evaluation map $\pi_w\colon\cX\ni x\mapsto x(w)\in\R$ is continuous.
Then, $\cX$ can be regarded as a projective limit in the following way.
First, let $\cV$ be the set of all finite dimensional linear subspaces of $W$, equipped with a partial order $\le$ simply given by the inclusion.
Next, for each $V\in\cV$, we define $\cY_V$ as the algebraic dual of $V$ equipped with the weak-* topology.
This makes $\cY_V$ a Hausdorff space.
Also, for any $V\le U\in\cV$, we define a continuous map $p_{VU}\colon\cY_U\to\cY_V$ by the restriction: $p_{VU}(L)\coloneqq L|_U$ for any $L\in\cY_V$.
Then, obviously, $(\cY_V,p_{VU})_{V\le U\in\cV}$ is a projective system.
Let $\widetilde\cY$ denote the projective limit of $(\cY_V,p_{VU})_{V\le U\in\cV}$.
Finally, we define a map $\Psi\colon\cX\to\widetilde\cY$ by $\Psi(L)\coloneqq(L|_V)_{V\in\cV}$ for any $L\in\cX$.
One can show that the map $\Psi$ is in fact homeomorphism using the consistency condition: every $(y_V)_{V\in\cV}\in\widetilde\cY$ satisfies that $y_V=p_{VU}(y_U)$ for any $V\le U\in\cV$~(see Theorem~4.6.9 in~\cite{DZ10} for details).
Consequently, the problem of finding an LDP in the topological space $\cX$ is transferred to that in the projective limit $\widetilde\cY$, which reduces to LDPs in finite dimensional linear subspaces of $W$ by Theorem~\ref{thm:DG}. 
Such application of Theorem~\ref{thm:DG} is summarized into the following useful theorem.
\begin{thm}[{\cite[Theorem~4.6.9]{DZ10}}]\label{thm:app_DG}
Let $W$ be a real vector space, and let $\cX$ be its algebraic dual $W'$ equipped with the weak-* topology.
Let $\{x^n\}_{n\in\N}$ be an $\cX$-valued process.
Assume that for any $m\in\N$ and $w_1,w_2,\ldots,w_m\in W$, the $\R^m$-valued process $\{(x^n(w_1),x^n(w_2),\ldots,x^n(w_m))\}_{n\in\N}$ satisfies the LDP with speed $a_n$ and a good rate function $J_{w_1,w_2,\ldots,w_m}\colon\R^m\to[0,\infty]$.
Then, the $\cX$-valued process $\{x^n\}_{n\in\N}$ satisfies the LDP with speed $a_n$ and good rate function $J\colon\cX\to[0,\infty]$ defined by
\[
J(x)\coloneqq\sup_{m\in\N}\sup_{w_1,w_2,\ldots,w_m\in W}J_{w_1,w_2,\ldots,w_m}((x(w_1),x(w_2),\ldots,x(w_m)))
\]
for any $x\in\cX$.
\end{thm}
In the proof of Theorem~\ref{thm:LDP_rdm_meas}, we also use the following basic lemma in the large deviation theory.
\begin{lem}[{\cite[Lemma~4.1.5~(b)]{DZ10}}]\label{lem:LDP_restr}
Let $\cE$ be a closed set of a topological space $\cX$, and let $\{S^n\}_{n\in\N}$ be an $\cX$-valued process such that $S^n\in\cE$ for every $n\in\N$.
If the $\cX$-valued process $\{S^n\}_{n\in\N}$ satisfies an LDP with speed $a_n$ and a good rate function $I'\colon\cX\to[0,\infty]$, then $\{S^n\}_{n\in\N}$ as an $\cE$-valued process satisfies the LDP with speed $a_n$ and good rate function $I'|_\cE\colon\cE\to[0,\infty]$.
\end{lem}

Combining Theorem~\ref{thm:app_DG} with Remark~\ref{rem:RMK} and Lemma~\ref{lem:LDP_restr}, we can prove Theorem~\ref{thm:LDP_rdm_meas} except for Statement~(2).
\begin{proof}[Proof of Theorem~\ref{thm:LDP_rdm_meas} except for Statement~(2)]
Note first that $(\Phi(\xi^n))(f)=L_{\xi^n}(f)=\xi^nf$ for every $f\in C_c(\Dl)$ from the definition of $\Phi$ in Remark~\ref{rem:RMK}.
Therefore, it follows from the assumption and Theorem~\ref{thm:app_DG} with $W=C_c(\Dl)$ that the $C_c(\Dl)'$-valued process $\{\Phi(\xi^n)\}_{n\in\N}$ satisfies the LDP with speed $a_n$ and good rate function $I'\colon C_c(\Dl)'\to[0,\infty]$ defined by
\[
I'(L)\coloneqq\sup_{m\in\N}\sup_{f_1,f_2,\ldots,f_m\in C_c(\Dl)}I_{f_1,f_2,\ldots,f_m}((L(f_1),L(f_2),\ldots,L(f_m)))
\]
for any $L\in C_c(\Dl)'$.
Since $C_c(\Dl)'_+\subset C_c(\Dl)'$ is a closed set as mentioned in Remark~\ref{rem:RMK}, Lemma~\ref{lem:LDP_restr} implies that the $C_c(\Dl)'_+$-valued process $\{\Phi(\xi^n)\}_{n\in\N}$ also satisfies the LDP with speed $a_n$ and good rate function $I'|_{C_c(\Dl)'_+}\colon C_c(\Dl)'_+\to[0,\infty]$.
Recalling that the map $\Phi\colon\cM(\Dl)\to C_c(\Dl)'_+$ is homeomorphism, we can conclude that the $\cM(\Dl)$-valued process $\{\xi^n\}_{n\in\N}$ satisfies the LDP with speed $a_n$ and the good rate function defined by~\eqref{eq:LDP_rdm_meas}.

Next, we prove Statement~(1).
The existence of a zero point of $I$ follows immediately from the goodness of $I$.
For the uniqueness of the zero point of $I$, suppose that $I(\xi)=I(\xi')=0$.
Then, $I_f(\xi f)=I_f(\xi'f)=0$ for any $f\in W$ by~\eqref{eq:LDP_rdm_meas}.
The uniqueness of the zero point of $I_f$ implies that $\xi f=\xi'f$, hence necessarily $\xi=\xi'$.
\end{proof}

In order to prove Theorem~\ref{thm:LDP_rdm_meas}~(2), we use a generalization of Theorem~\ref{thm:DZ_4.5.10} to the setting of topological vector spaces.
We first review some notion of topological vector space.
All vector spaces below are over the field of real numbers.
A \textit{topological vector space} $\cX$ is a vector space equipped with a topology such that the vector space operations are continuous, i.e.,
\begin{itemize}
\item the addition $\cX\times\cX\ni(x,x')\mapsto x+x'\in\cX$ is continuous,
\item the scalar multiplication $\R\times\cX\ni(\a,x)\mapsto\a x\in\cX$ is continuous.
\end{itemize}
A topological vector space $\cX$ is said to be \textit{locally convex} if there exists a local base at $0$ consisting of convex sets.
Given a topological vector space $\cX$, let $\cX^*\subset\cX'$ be the subspace consisting of all continuous linear functionals on $\cX$.
We refer to $\cX^*$ as the \textit{topological dual} of $\cX$.
For $x\in\cX$ and $\lm\in\cX'$, define $\langle\lm,x\rangle\coloneqq\lm(x)$ by convention.
The following theorem is useful to obtain a topology on a vector space that makes it a locally convex topological vector space.
\begin{thm}[{\cite[Theorem~3.10]{{Ru91}}}]\label{thm:loc_convex}
Let $\cX$ be a vector space, and let $\cH$ be a separating subspace of $\cX'$, i.e., $\cH\subset\cX'$ is a subspace satisfying that for any $0\neq x\in\cX$, there exists $\lm\in\cH$ such that $\langle\lm,x\rangle\neq0$.
Then, the $\cH$-topology makes $\cX$ into a locally convex Hausdorff topological vector space with $\cX^*=\cH$.
Here, $\cH$-topology is the weakest topology on $\cX$ such that every $\lm\in\cH$ is continuous.
\end{thm}
\begin{cor}\label{cor:loc_convex}
Let $W$ be a vector space, and let $\cX$ be its algebraic dual $W'$ equipped with the weak-* topology.
Then, $\cX$ is a locally convex Hausdorff topological vector space.
Moreover, $\cX^*$ and $W$ are isomorphic as vector spaces.
\end{cor}
\begin{proof}
We define an injective linear map $\iota\colon W\to\cX'$ by $\iota(w)\coloneqq\pi_w$ for any $w\in W$.
Recall here that $\pi_w$ is the evaluation map.
Set $\cH\coloneqq\im(\iota)$.
Then, the $\cH$-topology on $W'$ is nothing but the weak-* topology.
Additionally, we can easily verify that $\cH$ is a separating subspace of $\cX'$.
Therefore, $\cX$ is a locally convex Hausdorff topological vector space by Theorem~\ref{thm:loc_convex}.
Furthermore, $\cX^*=\cH$, which together with the injectivity of $\iota$ implies that $\cX^*$ and $W$ are isomorphic as vector spaces.
\end{proof}
The following theorem states that the good convex rate function for an LDP in a locally convex Hausdorff topological vector space is identified as the Fenchel--Legendre transform of the limiting logarithmic moment generating function (cf. Theorem~\ref{thm:DZ_4.5.10}).
\begin{thm}[{\cite[Theorem~4.5.10]{DZ10}}]\label{thm:DZ_4.5.10_general}
Let $\cX$ be a locally convex Hausdorff topological vector space, and let $\{S^n\}_{n\in\N}$ be an $\cX$-valued process.
Suppose that the $\cX$-valued process $\{S^n\}_{n\in\N}$ satisfies an LDP with speed $a_n$ and a good convex rate function $I\colon\cX\to[0,\infty]$, and also that
\[
\sup_{n\in\N}a_n^{-1}\log\E[\exp(a_n\langle\lm,S^n\rangle)]<\infty
\]
for any $\lm\in\cX^*$.
Then, for every $\lm\in\cX^*$, the limit
\[
\varphi(\lm)\coloneqq\lim_{n\to\infty}a_n^{-1}\log\E[\exp(a_n\langle\lm,S^n\rangle)]
\]
exists in $\R$, and $I=\varphi^*$ holds.
Here, $\varphi^*\colon\cX\to[0,\infty]$ is the Fenchel--Legendre transform of $\varphi\colon\cX^*\to[0,\infty]$$:$
\[
\varphi^*(x)\coloneqq\sup_{\lm\in\cX^*}\{\langle\lm,x\rangle-\varphi(\lm)\}
\]
for any $x\in\cX$.
\end{thm}

Now, we are ready to prove Theorem~\ref{thm:LDP_rdm_meas}~(2).
\begin{proof}[Proof of Theorem~\ref{thm:LDP_rdm_meas}~(2)]
Let $\cX$ denote the algebraic dual $C_c(\Dl)'$ with the weak-* topology.
From Corollary~\ref{cor:loc_convex} with $W=C_c(\Dl)$, $\cX$ is a locally convex Hausdorff topological vector space and $\cX^*=\im(\iota)$.
Here, $\iota$ is the injective linear map defined in the proof of Corollary~\ref{cor:loc_convex}.
As mentioned in \textit{Proof of Theorem~\ref{thm:LDP_rdm_meas} except for~(2)}, the $C_c(\Dl)'$-valued process $\{\Phi(\xi^n)\}_{n\in\N}$ satisfies the LDP with speed $a_n$ and a good rate function $I'\colon C_c(\Dl)'\to[0,\infty]$ defined by
\[
I'(L)\coloneqq\sup_{m\in\N}\sup_{f_1,f_2,\ldots,f_m\in C_c(\Dl)}I_{f_1,f_2,\ldots,f_m}((L(f_1),L(f_2),\ldots,L(f_m)))
\]
for any $L\in C_c(\Dl)'$.

We next claim that $I'$ is a convex function from the assumption of Statement~(2).
Indeed, suppose that $L_1,L_2\in C_c(\Dl)'$ and $t\in(0,1)$, and write $L'\coloneqq tL_1+(1-t)L_2$.
Then, for any $m\in\N$ and $f_1,f_2,\ldots,f_m\in C_c(\Dl)$,
\begin{align*}
&I_{f_1,f_2,\ldots,f_m}((L'(f_1),L'(f_2),\ldots,L'(f_m)))\\
&=I_{f_1,f_2,\ldots,f_m}(t(L_1(f_1),L_1(f_2),\ldots,L_1(f_m))+(1-t)(L_2(f_1),L_2(f_2),\ldots,L_2(f_m)))\\
&\le tI_{f_1,f_2,\ldots,f_m}((L_1(f_1),L_1(f_2),\ldots,L_1(f_m)))+(1-t)I_{f_1,f_2,\ldots,f_m}((L_2(f_1),L_2(f_2),\ldots,L_2(f_m)))\\
&\le tI'(L_1)+(1-t)I'(L_2),
\end{align*}
which implies that $I'(L')\le tI'(L_1)+(1-t)I'(L_2)$.
Consequently, $I'$ is convex.
Furthermore, for any $\lm=\iota(f)\in\im(\iota)=\cX^*$,
\[
\sup_{n\in\N}a_n^{-1}\log\E[\exp(a_n\langle\lm,\Phi(\xi^n)\rangle)]
=\sup_{n\in\N}a_n^{-1}\log\E[\exp(a_n\xi^nf)]
<\infty
\]
from the assumption.
Therefore, Theorem~\ref{thm:DZ_4.5.10_general} implies that for every $\lm\in\cX^*$, the limit
\[
\varphi(\lm)\coloneqq\lim_{n\in\N}a_n^{-1}\log\E[\exp(a_n\langle\lm,\Phi(\xi^n)\rangle)]
\]
exists in $\R$, and it holds that
\[
I'(L)=
\sup_{\lm\in\cX^*}\{\langle\lm,L\rangle-\varphi(\lm)\}
=\sup_{f\in C_c(\Dl)}\{L(f)-\varphi(\lm)\}.
\]
for any $L\in\cX$.
Consequently, using Lemma~\ref{lem:LDP_restr} with the closed set $C_c(\Dl)'_+\subset C_c(\Dl)'$ and the fact that the map $\Phi\colon\cM(\Dl)\to C_c(\Dl)'_+$ is homeomorphism, we can conclude that the $\cM(\Dl)$-valued process $\{\xi^n\}_{n\in\N}$ satisfies the LDP with speed $a_n$ and the good rate function defined by~\eqref{eq:LDP_rdm_meas_1}.
\end{proof}

\section*{Acknowledgements}
The authors would like to thank Ryuya Namba, Tomoyuki Shirai, and Ryokichi Tanaka for useful discussions and valuable comments.
The first author is supported by a JSPS Grant-in-Aid for Transformative Research Areas (A) (22A201) and a JSPS Grant-in-Aid for Scientific Research (A) (JP20H00119).
The fourth author is supported by a JSPS Grant-in-Aid for Early-Career Scientists Grant Number (18K13426) and a JSPS Grant-in-Aid for Early-Career Scientists Grant Number (22K13929).



\begin{thebibliography}{99}
\bibitem{AL15}
L. Aronshtam and N. Linial, \textit{When does the top homology of a random simplicial complex vanish?}, Random Structures Algorithms \textbf{46} (2015), no. 1, 26--35. https://doi.org/10.1002/rsa.20495. MR3291292
%
\bibitem{AL16}
L. Aronshtam and N. Linial, \textit{The threshold for $d$-collapsibility in random complexes*}, Random Structures Algorithms \textbf{48} (2016), no. 2, 260--269. https://doi.org/10.1002/rsa.20585. MR3449598
%
\bibitem{ALLM13}
L. Aronshtam, N. Linial, T. \L uczak, and R. Meshulam, \textit{Collapsibility and vanishing of top homology in random simplicial complexes}, Discrete Comput. Geom. \textbf{49} (2013), no. 2, 317--334. https://doi.org/10.1007/s00454-012-9483-8. MR3017914
%
\bibitem{BHK11}
E. Babson, C. Hoffman, and M. Kahle, \textit{The fundamental group of random $2$-complexes}, J. Amer. Math. Soc. \textbf{24} (2011), no. 1, 1--28. MR2726597
%
\bibitem{BK18}
O. Bobrowski and M. Kahle, \textit{Topology of random geometric complexes: a survey}, J. Appl. Comput. Topol. \textbf{1} (2018), no. 3--4, 331--364. https://doi.org/10.1007/s41468-017-0010-0. MR3975557
%
\bibitem{BK22}
O. Bobrowski and D. Krioukov, \textit{Random simplicial complexes: models and phenomena}, Higher-order systems, 59--96, Underst. Complex Syst., Springer, Cham, 2022. https://doi.org/10.1007/978-3-030-91374-8\_2. MR4433790
%
\bibitem{BM15}
O. Bobrowski and S. Mukherjee, \textit{The topology of probability distributions on manifolds}, Probab. Theory Related Fields \textbf{161} (2015), no. 3--4, 651--686. https://doi.org/10.1007/s00440-014-0556-x. MR3334278
%
\bibitem{CCFK11}
D. Cohen, A. Costa, M. Farber, and T. Kappeler, \textit{Topology of random $2$-complexes}, Discrete Comput. Geom. \textbf{47} (2012), no. 1, 117--149. https://doi.org/10.1007/s00454-011-9378-0. MR2886093
%
\bibitem{CEH07}
D. Cohen-Steiner, H. Edelsbrunner, and J. Harer, \textit{Stability of persistence diagrams}, Discrete Comput. Geom. \textbf{37} (2007), no. 1, 103--120. https://doi.org/10.1007/s00454-006-1276-5. MR2279866
%
\bibitem{CF16a}
A. Costa and M. Farber, \textit{Large random simplicial complexes, I}, J. Topol. Anal. \textbf{8} (2016), no. 3, 399--429. https://doi.org/10.1142/S179352531650014X. MR3509567
%
\bibitem{CF16b}
A. Costa and M. Farber, \textit{Random simplicial complexes}, Configuration Spaces, Springer INdAM Series, vol.~\textbf{14}, Springer, Cham, 2016, 129--153. MR3615731
%
\bibitem{CF17a}
A. Costa and M. Farber, \textit{Large random simplicial complexes, II; the fundamental group}, J. Topol. Anal. \textbf{9} (2017), no. 3, 441--483. https://doi.org/10.1142/S1793525317500170. MR3661651
%
\bibitem{CF17b}
A. Costa and M. Farber, \textit{Large random simplicial complexes, III: the critical dimension}, J. Knot Theory Ramifications \textbf{26} (2017), no. 2, 1740010. https://doi.org/10.1142/S0218216517400107. MR3604492
%
\bibitem{DZ10}
A. Dembo and O. Zeitouni, \textit{Large deviations techniques and applications}, Corrected reprint of the second (1998) edition, Stochastic Modelling and Applied Probability, Vol.~\textbf{38}, Springer-Verlag, Berlin, 2010. https://doi.org/10.1007/978-3-642-03311-7. MR2571413
%
\bibitem{DL22}
K. A. Dowling and E. Lundberg, \textit{Homotopy types of random cubical complexes}, J. Appl. Comput. Topol. \textbf{6} (2022), no. 1, 1--26. https://doi.org/10.1007/s41468-021-00082-2. MR4392063
%
\bibitem{ELZ02}
H. Edelsbrunner, D. Letscher, and A. Zomorodian, \textit{Topological persistence and simplification}, Discrete Comput. Geom. \textbf{28} (2002), no. 4, 511--533. https://doi.org/10.1007/s00454-002-2885-2. MR1949898
%
\bibitem{ER59}
P. Erd\H os and A. R\'enyi, \textit{On random graphs. I}, Publ. Math. Debrecen \textbf{6} (1959), 290--297. MR0120167
%
\bibitem{ER60}
P. Erd\H os and A. R\'enyi, \textit{On the evolution of random graphs}, Publ. Math. Inst. Hungarian Acad. Sci. \textbf{5} (1960), 17--61. MR0125031
%
\bibitem{FMN22}
M. Farber, L. Mead, and T. Nowik, \textit{Random simplicial complexes, duality and the critical dimension}, J. Topol. Anal. \textbf{14} (2022), no. 1, 1--31. https://doi.org/10.1142/S1793525320500387. MR4411098
%
\bibitem{Fo19}
C. F. Fowler, \textit{Homology of multi-parameter random simplicial complexes}, Discrete Comput. Geom. \textbf{62} (2019), no. 1, 87--127. https://doi.org/10.1007/s00454-018-00056-9. MR3959923
%
\bibitem{GTT19}
A. Goel, K. D. Trinh, and K. Tsunoda, \textit{Strong law of large numbers for Betti numbers in the thermodynamic regime}, J. Stat. Phys. \textbf{174} (2019), no. 4, 865--892. https://doi.org/10.1007/s10955-018-2201-z. MR3913900
%
\bibitem{HK19}
M. Hino and S. Kanazawa, \textit{Asymptotic behavior of lifetime sums for random simplicial complex processes}, J. Math. Soc. Japan \textbf{71} (2019), no. 3, 765--804. https://doi.org/10.2969/jmsj/79777977. MR3984242
%
\bibitem{HS16}
Y. Hiraoka and T. Shirai, \textit{Tutte polynomials and random-cluster models in Bernoulli cell complexes}, RIMS K\^oky\^uroku Bessatsu \textbf{B59} (2016), 289--304. MR3675939
%
\bibitem{HST18}
Y. Hiraoka, T. Shirai, and K. D. Trinh, \textit{Limit theorems for persistence diagrams}, Ann. Appl. Probab. \textbf{28} (2018), no. 5, 2740--2780. https://doi.org/10.1214/17-AAP1371. MR3847972
%
\bibitem{HT18}
Y. Hiraoka and K. Tsunoda, \textit{Limit theorems for random cubical homology}, Discrete Comput. Geom. \textbf{60} (2018), 665--687. https://doi.org/10.1007/s00454-018-0007-z. MR3849145
%
\bibitem{HO22}
C. Hirsch and T. Owada, \textit{Large deviation principle for geometric and topological functionals and associated point processes}, arXiv:2201.07276, 2022.
%
\bibitem{HKP17}
C. Hoffman, M. Kahle, and E. Paquette, \textit{The threshold for integer homology in random $d$-complexes}, Discrete Comput. Geom. \textbf{57} (2017), no. 4, 810--823. https://doi.org/10.1007/s00454-017-9863-1. MR3639605
%
\bibitem{KMM04}
T. Kaczynski, K. Mischaikow, and M. Mrozek, \textit{Computational homology}, Applied Mathematical Sciences, Vol.~\textbf{157}, Springer-Verlag, New York, 2004. https://doi.org/10.1007/b97315. MR2028588
%
\bibitem{K11}
M. Kahle, \textit{Random geometric complexes}, Discrete Comput. Geom. \textbf{45} (2011), no. 3, 553--573. https://doi.org/10.1007/s00454-010-9319-3. MR2770552
%
\bibitem{K14b}
M. Kahle, \textit{Topology of random simplicial complexes: a survey}, Algebraic Topology: Applications and New Directions, 201--221, Contemp. Math., \textbf{620}, Amer. Math. Soc., Providence, RI, 2014. http://dx.doi.org/10.1090/conm/620/12367. MR3290093
%
\bibitem{KPR21}
M. Kahle, E. Paquette, and E. Rold\'an, \textit{Topology of random $2$-dimensional cubical complexes}, Forum Math. Sigma \textbf{9} (2021), E76. https://doi.org/10.1017/fms.2021.64. MR4345012
%
\bibitem{KP16}
M. Kahle and B. Pittel, \textit{Inside the critical window for cohomology of random $k$-complexes}, Random Structures Algorithms \textbf{48} (2016), no. 1, 102--124. https://doi.org/10.1002/rsa.20577. MR3432573
%
\bibitem{Kal17}
O. Kallenberg, \textit{Random measures, theory and applications}, Probability Theory and Stochastic Modelling, Vol.~\textbf{77}, Springer, Cham, 2017. https://doi.org/10.1007/978-3-319-41598-7. MR3642325
%
\bibitem{Ka22}
S. Kanazawa, \textit{Law of large numbers for Betti numbers of homogeneous and spatially independent random simplicial complexes}, Random Structures Algorithms \textbf{60} (2022), no. 1, 68--105. https://doi.org/10.1002/rsa.21015. MR4340474
%
\bibitem{Ko10}
D. N. Kozlov, \textit{The threshold function for vanishing of the top homology group of random $d$-complexes}, Proc. Amer. Math. Soc. \textbf{138} (2010), no. 12, 4517--4527. MR2680076
%
\bibitem{LM06}
N. Linial and R. Meshulam, \textit{Homological connectivity of random $2$-complexes}, Combinatorica \textbf{26} (2006), no. 4, 475--487. https://doi.org/10.1007/s00493-006-0027-9. MR2260850
%
\bibitem{LP16}
N. Linial and Y. Peled, \textit{On the phase transition in random simplicial complexes}, Ann. of Math.~(2) \textbf{184} (2016), no. 3, 745--773. https://doi.org/10.4007/annals.2016.184.3.3. MR3549622
%
\bibitem{LP17}
N. Linial and Y. Peled, \textit{Random simplicial complexes: around the phase transition}, A Journey Through Discrete Mathematics, Springer, Cham, 2017, 543--570. MR3726613
%
\bibitem{LuP18}
T. \L uczak and Y. Peled, \textit{Integral homology of random simplicial complexes}, Discrete Comput. Geom. \textbf{59} (2018), no. 1, 131--142. https://doi.org/10.1007/s00454-017-9938-z. MR3738338
%
\bibitem{MW09}
R. Meshulam and N. Wallach, \textit{Homological connectivity of random $k$-dimensional complexes}, Random Structures Algorithms \textbf{34} (2009), no. 3, 408--417. https://doi.org/10.1002/rsa.20238. MR2504405
%
\bibitem{NP18}
A. Newman and E. Paquette, \textit{The integer homology threshold in $Y_d(n,p)$}, arXiv:1808.10647, 2018.
%
\bibitem{OST21}
T. Owada, G. Samorodnitsky, and G. Thoppe, \textit{Limit theorems for topological invariants of the dynamic multi-parameter simplicial complex}, Stochastic Process. Appl. \textbf{138} (2021), 56--95. https://doi.org/10.1016/j.spa.2021.04.008. MR4252193
%
\bibitem{P03}
M. Penrose, \textit{Random geometric graphs}, Oxford Studies in Probability, Vol.~\text{5}. Oxford University Press, Oxford, 2003. https://doi.org/10.1093/acprof:oso/9780198506263.001.0001. MR1986198
%
\bibitem{Ru91}
W. Rudin, \textit{Functional analysis}, Second edition, International Series in Pure and Applied Mathematics, McGraw-Hill, Inc., New York, 1991. MR1157815
%
\bibitem{SO22}
G. Samorodnitsky and T. Owada, \textit{Large deviations for subcomplex counts and Betti numbers in multi-parameter simplicial complexes}, arXiv:2202.08330, 2022.
%
\bibitem{SY01}
T. Sepp\"al\"ainen and J. Yukich, \textit{Large deviation principles for Euclidean functionals and other nearly additive processes}, Probab. Theory Related Fields \textbf{120} (2001), no. 3, 309--345. https://doi.org/10.1007/PL00008785. MR1843178
%
\bibitem{WW16}
M. Werman and M. L. Wright, \textit{Intrinsic volumes of random cubical complexes}, Discrete Comput. Geom. \textbf{56} (2016), no. 1, 93--113. https://doi.org/10.1007/s00454-016-9789-z. MR3509032
%
\bibitem{YSA17}
D. Yogeshwaran, E. Subag, and R. J. Adler, \textit{Random geometric complexes in the thermodynamic regime}, Probab. Theory Related Fields \textbf{167} (2017), no. 1--2, 107--142. https://doi.org/10.1007/s00440-015-0678-9. MR3602843
%
\bibitem{ZC05}
A. Zomorodian and G. Carlsson, \textit{Computing persistent homology}, Discrete Comput. Geom. \textbf{33} (2005), no. 2, 249--274. https://doi.org/10.1007/s00454-004-1146-y. MR2121296
\end{thebibliography}
\end{document}